\documentclass[12pt,twoside]{amsart}
\usepackage{latexsym,amsfonts,amsmath,amssymb,soul}
\usepackage{enumerate,soul}
\usepackage[titletoc]{appendix}
\usepackage[usenames, dvipsnames]{color}
\numberwithin{equation}{section}

\makeatletter
\@namedef{subjclassname@2020}{%
  \textup{2020} Mathematics Subject Classification}
\makeatother

\newtheorem{Th}{Theorem}[section]
\newtheorem{Rem}[Th]{Remark}

\newtheorem{Lemma}[Th]{Lemma}
\newtheorem{Def}[Th]{Definition}
\newtheorem{Prop}[Th]{Proposition}
\newtheorem{Cor}[Th]{Corollary}

\catcode`\@=11

\renewcommand{\section}%
   {\setcounter{equation}{0}\@startsection {section}{1}{\z@}{-3.5ex plus -1ex
  minus -.2ex}{2.3ex plus .2ex}{\Large\bf}}

\def\WF{\mathop{\rm WF}\nolimits}
\def\supp{\mathop{\rm supp}\nolimits}
\def\ds{\displaystyle}
\def\R{\mathbb R}
\def\C{\mathbb C}
\def\N{\mathbb N}

\def\Z{\mathbb Z}

\newcommand{\F}{\mathcal{F}}

\newcommand{\Sch}{\mathcal{S}}
\newcommand{\calO}{\mathcal{O}}

\newcommand{\afrac}[2]{\genfrac{}{}{0pt}{1}{#1}{#2}}

\newcommand{\beqsn}{\arraycolsep1.5pt\begin{eqnarray*}}
\newcommand{\eeqsn}{\end{eqnarray*}\arraycolsep5pt}
\newcommand{\beqs}{\arraycolsep1.5pt\begin{eqnarray}}
\newcommand{\eeqs}{\end{eqnarray}\arraycolsep5pt}

\title{Stability of global wave front sets by perturbations of frames}

\author[Boiti]{Chiara Boiti}
\address{
Dipartimento di Matematica e Informatica \\Universit\`a di Ferrara\\
Via Ma\-chia\-vel\-li n.~30\\
I-44121 Ferrara\\
Italy}
\email{chiara.boiti@unife.it}

\author[Jornet]{David Jornet}
\address{
Instituto Universitario de Matem\'atica Pura y Aplicada IUMPA\\
Universitat Po\-li\-t\`ecni\-ca de Val\`encia\\
Camino de Vera, s/n\\
E-46022 Val\`encia\\
Spain}
\email{djornet@mat.upv.es}

\author[Oliaro]{Alessandro Oliaro}
\address{Dipartimento di Matematica\\ Universit\`a di Torino\\
 Via Carlo Alberto n.~10\\ I-10123 Torino\\ Italy}
 \email{alessandro.oliaro@unito.it}

\begin{document}



\begin{abstract}
In this paper we consider the Gabor wave front set of ultradistributions in the frame of ultradifferentiable functions. We prove that such a wave front set, defined through a Gabor frame on a regular lattice, is not affected by perturbations of the frame, in two different cases: when we consider $\varepsilon$-perturbations of Christensen type, and when we consider nonstationary Gabor frames.
\end{abstract}

\maketitle


\markboth{\sc  Stability of global wave front sets by perturbations of frames}{\sc C.~Boiti, D.~Jornet, A.~Oliaro}

\section{Introduction}
\label{sec1}

The wave front set is a very important concept in the study of local behavior of distributions, since it locates the singularities of the distribution together with the directions of the high frequencies that are responsible for those singularities. Such an analysis of the frequencies is done by looking at the decay in different cones on the Fourier transform side. In the context of Schwartz distributions, the wave front set was originally introduced by H\"ormander \cite{H5}. Since then, a huge literature has been produced on wave front sets and the corresponding applications to the study of propagation of singularities for linear partial differential operators in spaces of distributions and ultradistributions in a local sense, see for instance \cite{AJO1,AJO2,BJ-wfs,BJJ,FGJ3,Ro} and the references therein.

Several versions of the wave front set are defined and studied in the literature; here we focus in particular on this concept in global classes of functions and distributions. In the Schwartz class $\mathcal{S}$ and in the corresponding tempered distribution space $\mathcal{S}'$, for instance, the concept of singular support does not make sense, and the regularity is related with the behavior in the whole space $\mathbb{R}^d$. However, it is possible to define a global wave front set to analyze the micro-regularity of a distribution, and the basic idea is that here the cones should be taken in the whole of the phase space variables. Two types of global wave front set were already introduced by H\"ormander \cite{H3}, who studied the $C^\infty$ wave front set for tempered distributions, in the Beurling setting, and the analytic wave front set for ultradistributions of Gelfand-Shilov type, in the Roumieu setting, with the aim of analyzing quadratic hyperbolic operators. These global versions of the wave front set have been almost ignored in the literature for several years, and they have received more attention in the last years, starting from the paper \cite{RW}; in the latter work, indeed, the global wave front set is studied in relation with more recent tools from time-frequency analysis, like the Gabor transform and the Gabor frames. Such tools have been introduced as a way to analyze signals from the point of view of the joint energy distribution with respect to time and frequency, and they have found applications in the analysis of partial differential equations and pseudo differential operators, see for instance \cite{BJO-Regularity,CR1,G,MO,NR} and the references therein. This makes the Gabor transform to be a good tool also for the global wave front set, since it permits in a natural way the analysis in cones in the whole of phase space variables. In fact, \cite{RW} prove that the H\"ormander $C^\infty$ global wave front set can be equivalently defined by means of the Gabor transform, as well as by means of Gabor frames (which is called ``Gabor wave front set''). We refer also to \cite{N} and \cite{SW2}, where the homogeneous wave front set is introduced, and is shown to coincide with the Gabor wave front set. The global wave front set in the frame of Gelfand-Shilov ultradistributions of Gevrey type is studied in \cite{CS}; moreover, other versions of global wave front set in the spirit of time-frequency analysis have been analyzed, see for instance \cite{RW1,W} for the anisotropic case, and \cite{As,CR2} for the use of Wigner type transforms in the definition of wave front set. We consider here the case of global (Gabor) wave front sets in the frame of ultradifferentiable classes of Beurling type $\mathcal{S}_\omega(\R^d)$, where $\omega$ is a weight function in the sense of Braun, Meise and Taylor \cite{BMT}. 

The space $\mathcal{S}_\omega(\R^d)$ was originally introduced by Bj\"orck \cite{Bjorck} and is defined as the set of all functions belonging to $L^1(\R^d)$ such that (its Fourier transform is in $L^1(\R^d)$ and) for every $\lambda>0$ and $\alpha\in\N^d_0$,
$$
\|e^{\lambda\omega}\partial^\alpha u\|_\infty<+\infty,\qquad \|e^{\lambda\omega}\partial^\alpha \hat{u}\|_\infty<+\infty,
$$
where $\N_0:=\N\cup\{ 0\}$. This space is invariant under Fourier transform and coincides with the classical Schwartz space $\mathcal{S}(\R^d)$ when $\omega(t)=\log (1+t)$ for $t\geq 0$. Ultradifferentiable classes constitute then a large scale of spaces that are suitable both for microlocal analysis, pseudodifferential operators (see for instance \cite{A-Q,AJ,ABJO-global}, or \cite{PP,Pr} for the case of spaces defined by sequences), and time-frequency analysis (see for instance \cite{GZ}, or \cite{ABJO-cptWeyl,boiti2020nuclearity}). The space $\mathcal{S}_\omega(\R^d)$ coincides with the Gelfand-Shilov space of Beurling type and order $s>1$ when $\omega(t)=t^{1/s}$ (Gevrey weight) and we consider in this paper the non quasianalytic case, so that $\mathcal{S}_\omega(\R^d)$ contains non trivial compactly supported functions. The ultradifferentiable version of the Gabor wave front set considered in \cite{CS,H3,RW} is introduced in the setting of Beurling ultradistributions in $\mathcal{S}_\omega'(\R^d)$ in \cite{BJO-Gabor}, where it is shown that it can be equivalently described by means of the Gabor transform and of Gabor frames for subadditive weights, and it is applied to the study of global regularity of (pseudo)differential operators of infinite order. Moreover, taking advantage to the study of global pseudodifferential operators and existence of parametrices for different quantizations in \cite{A-Q,AJ}, in \cite{ABJO-global} the Gabor ultradifferentiable wave front set is analyzed in connection with pseudodifferential calculus, showing that it can be equivalently defined in terms of Weyl quantizations when the weight is smaller than some Gevrey weight. In \cite{ABJO-global}, moreover, applications to the study of propagation of singularities for Weyl quantizations with respect to  Weyl wave front sets are given in the general setting of \cite{AJ}.

In this paper we continue the study of global wave front sets in the ultradifferentiable setting, focusing in particular on its definition related with Gabor frames, cf. Definition \ref{def33Gabor} below. Recall that a Gabor frame is the set of translations and modulations of a single function, where the translations and modulations are taken in a lattice. A natural question is whether the request that the frame is of this type is essential or can be relaxed. The purpose of this paper is to show that frames that are, in some sense, not too ``far away'' from Gabor type, give equivalent definitions of the global (Gabor) wave front set. This is done, in this paper, in two different cases. First of all we study the case of frames that are \emph{perturbations} of a Gabor frame in the sense of \cite{C}. This means that we consider frames that do not have a Gabor structure, but for which each element has sufficiently small distance from a corresponding element of a Gabor frame; in this case, the Gabor wave front set does not depend on the perturbation. Secondly, we consider the case of \emph{nonstationary} Gabor frames in the sense of \cite{Balazs}, where more freedom is admitted in the choice of the lattice, and we may have less regular structure with respect to classical Gabor frames. Also in this case, we show that the wave front set is not affected by the loss of structure of the frame.

A general question could be to characterize the frames producing a wave front set that coincides with the Gabor wave front set defined by the Gabor transform, or equivalently by classical Gabor frames. This seems a difficult problem, and remains at the moment an open question.

The structure of the paper is the following. In Section \ref{sec2} we collect the main definitions and the basic results that are needed in the following. Section \ref{sec3} is devoted to the case of perturbations of frames in the sense of \cite{C}, and in Section \ref{sec4} nonstationary Gabor frames are analyzed.

\section{Notation and preliminary results}
\label{sec2}

For a function $f\in L^1(\R^d)$, the {\em Fourier transform} of $f$ is defined by
\beqsn
\F(f)(\xi)=\hat{f}(\xi)=\int_{\R^d} f(x)e^{-ix\cdot \xi}dx,
\eeqsn
with standard extensions to more general spaces of functions or distributions.

We denote by $T_x$, $M_\xi$ and $\Pi(z)$ the {\em translation}, {\em modulation} and {\em phase-space shift} operators defined by
\beqsn
&&T_xf(t)=f(t-x), \quad M_\xi f(t)=e^{it\cdot\xi}f(t),\\
&&\Pi(z)f(t)=M_\xi T_xf(t)=e^{it\cdot\xi}f(t-x),
\eeqsn
for $t,x,\xi\in\R^d$, $z=(x,\xi)$.

For a window function $\varphi\in L^2(\R^d)\setminus\{0\}$, the {\em short-time Fourier
transform} of $f\in L^2(\R^d)$ is defined by
\beqsn
V_\varphi f(z)=\langle f,\Pi(z)\varphi\rangle
=\int_{\R^d}f(t)\bar{\varphi}(t-x)e^{-it\cdot\xi}dt,
\qquad z=(x,\xi)\in\R^{2d},
\eeqsn
where $\langle\cdot,\cdot\rangle$ is the $L^2$-inner product,
with standard extensions to more general spaces of functions or distributions
(with the duality conjugate linear product).

We shall work here in the classes of rapidly decreasing ultradifferentiable functions for a weight function $\omega$ as in the following:

\begin{Def}
\label{def21Gabor}
A {\em non-quasianalytic subadditive weight function} is a continuous increasing function
$\omega: [0,+\infty)\to[0,+\infty)$ satisfying
\begin{itemize}
\item[$(\alpha)$]
$\omega(t_1+t_2)\leq\omega(t_1)+\omega(t_2),\quad
\forall t_1,t_2\geq0$;
\item[$(\beta)$]
$\ds \int_1^{+\infty}\frac{\omega(t)}{t^2}dt<+\infty$;
\item[$(\gamma)$]
$\exists a\in\R$, $b>0$ $s.t.$ $\omega(t)\geq a+b\log(1+t),\quad\forall t\geq0$;
\item[$(\delta)$]
$\varphi(t):=\omega(e^t)$ is convex.
\end{itemize}
We then define $\omega(\zeta):=\omega(|\zeta|)$ for $\zeta\in \C^d$, where $|\cdot|$ is the
Euclidean norm in $\C^d$.
\end{Def}


The {\em Young conjugate} $\varphi^*$ of $\varphi$ is then defined by
\beqsn
\varphi^*(s):=\sup_{t\geq0}(st-\varphi(t)),\qquad s\geq0.
\eeqsn

\begin{Def}
\label{defSomega}
We define $\Sch_\omega(\R^d)$ as the set of all $f\in\Sch(\R)$ such that,
for all $\lambda>0,\alpha\in\N_0^d$ (here $\N_0=\N\cup\{0\}$):
\beqsn
\sup_{x\in\R^d}|D^\alpha f(x)|e^{\lambda\omega(x)}<+\infty,
\quad \sup_{\xi\in\R^d}|D^\alpha \hat f(\xi)|e^{\lambda\omega(\xi)}<+\infty.
\eeqsn
We denote by $\Sch'_\omega(\R^d)$ the strong dual space of $\Sch_\omega(\R^d)$.
\end{Def}

In \cite[Thms.~2.4, 2.5]{BJO-realPW} we provide the space $\Sch_\omega(\R^d)$ with
different equivalent systems of seminorms such as, for example:
\beqsn
&&p_{\lambda,\mu}(f)=\sup_{\alpha,\beta\in\N_0^d}\sup_{x\in\R^d}|x^\beta D^\alpha f(x)|
e^{-\lambda\varphi^*\left(\frac{|\alpha|}{\lambda}\right)-\mu\varphi^*\left(\frac{|\beta|}{\mu}\right)}\\
&&q_{\lambda,\mu}(f)=\sup_{\alpha\in\N_0^d}\sup_{x\in\R^d}|D^\alpha f(x)|
e^{-\lambda\varphi^*\left(\frac{|\alpha|}{\lambda}\right)+\mu\omega(x)}\\
&&r_\lambda(f)=\sup_{z\in\R^{2d}}|V_\psi f(z)|e^{\lambda\omega(z)}
\eeqsn
for a window function $\psi\in\Sch_\omega(\R^d)\setminus\{0\}$.

Let $\varphi\in\Sch_\omega(\R^d)\setminus\{0\}$ and $\Lambda=\alpha\Z^d\times\beta\Z^d$
a lattice with $\alpha,\beta>0$ sufficiently small so that
$\{\Pi(\sigma)\varphi\}_{\sigma\in\Lambda}$ is a Gabor frame in $L^2(\R^d)$, i.e.
the sequence $\{x_\sigma\}_{\sigma\in\Lambda}=\{\Pi(\sigma)\varphi\}_{\sigma\in\Lambda}$ is a frame in the Hilbert space $L^2(\R^d)$ for some {\em lower} and {\em upper frame bounds} $A,B>0$:
\beqs
\label{2}
A\|f\|^2\leq\sum_{\sigma\in\Lambda}|\langle f,\Pi(\sigma)\varphi\rangle|^2
\leq B\|f\|^2,
\qquad\forall f\in L^2(\R^d),
\eeqs
where $\|\cdot\|$ denotes the $L^2$-norm.
Recall that if the second inequality of \eqref{2} is satisfied, then 
$\{\Pi(\sigma)\varphi\}_{\sigma\in\Lambda}$ is said to be a {\em Bessel sequence}.

Note that $\langle f,\Pi(\sigma)\varphi\rangle=V_\varphi f(\sigma)$. In \cite{BJO-Gabor}
we defined the following global wave front set with respect to the short-time Fourier transform:
\begin{Def}
\label{def33Gabor}
If $u\in \Sch'_\omega(\R^d)$, we say that $z_0\in\R^{2d}\setminus\{0\}$ is not in the
{\em Gabor $\omega$-wave front set} $\WF^G_\omega(u)$ of $u$ if there exists an open conic
set $\Gamma\subseteq\R^{2d}\setminus\{0\}$ containing $z_0$ such that
\beqs
\label{39Gabor}
\sup_{\sigma\in\Lambda\cap\Gamma}e^{\lambda\omega(\sigma)}
|\langle u,\Pi(\sigma)\varphi\rangle|<+\infty,
\qquad\forall\lambda>0.
\eeqs
\end{Def}

We recall that $\WF^G_\omega(u)$ does not depend on the choice of the window function $\varphi$ by \cite[Prop.~3.2]{BJO-Gabor}.
It will be useful, in the sequel,  to write condition \eqref{39Gabor} in a different equivalent form:
\begin{Lemma}
\label{lemma1}
Condition \eqref{39Gabor} is equivalent to
\beqs
\label{1}
\sum_{\sigma\in\Lambda\cap\Gamma}e^{\lambda\omega(\sigma)}
|\langle u,\Pi(\sigma)\varphi\rangle|^2<+\infty,
\qquad\forall\lambda>0.
\eeqs
\end{Lemma}

\begin{proof}
Assume that
\beqsn
\forall\lambda>0\ \exists C_\lambda>0:\ \quad
\sup_{\sigma\in\Lambda\cap\Gamma}e^{\lambda\omega(\sigma)}
|\langle u,\Pi(\sigma)\varphi\rangle|\leq C_\lambda
\eeqsn
and prove \eqref{1}. Indeed,
\beqsn
\sum_{\sigma\in\Lambda\cap\Gamma}e^{\lambda\omega(\sigma)}
|\langle u,\Pi(\sigma)\varphi\rangle|^2
=\sum_{\sigma\in\Lambda\cap\Gamma}e^{-\lambda\omega(\sigma)}
e^{2\lambda\omega(\sigma)}
|\langle u,\Pi(\sigma)\varphi\rangle|^2
\leq C_{\lambda}^2\sum_{\sigma\in\Lambda\cap\Gamma}e^{-\lambda\omega(\sigma)}
\leq C'_\lambda
\eeqsn
for some $C'_\lambda>0$, since the last series converges for $\lambda\geq\lambda_0$
sufficiently large by condition $(\gamma)$ on the weight $\omega$. If $\lambda<\lambda_0$
we increase 
$e^{\lambda\omega(\sigma)}|\langle u,\Pi(\sigma)\varphi\rangle|^2$ by $e^{\lambda_0\omega(\sigma)}
|\langle u,\Pi(\sigma)\varphi\rangle|^2$. Hence \eqref{1} is proved.

Let us now prove the opposite implication. Assume that
\beqsn
\forall\lambda>0\ \exists C_\lambda>0:\ \quad
\sum_{\sigma\in\Lambda\cap\Gamma}e^{\lambda\omega(\sigma)}
|\langle u,\Pi(\sigma)\varphi\rangle|^2\leq C_\lambda
\eeqsn
and prove \eqref{39Gabor}.
Indeed,
\beqsn
e^{2\lambda\omega(\sigma)}
|\langle u,\Pi(\sigma)\varphi\rangle|^2
\leq \sum_{\sigma\in\Lambda\cap\Gamma}e^{2\lambda\omega(\sigma)}
|\langle u,\Pi(\sigma)\varphi\rangle|^2\leq C_{2\lambda},
\qquad\forall\sigma\in\Lambda\cap\Gamma,
\eeqsn
implies
\beqsn
\sup_{\sigma\in\Lambda\cap\Gamma}e^{2\lambda\omega(\sigma)}
|\langle u,\Pi(\sigma)\varphi\rangle|^2\leq C_{2\lambda},
\eeqsn
which is equivalent to
\beqsn
\left(\sup_{\sigma\in\Lambda\cap\Gamma}e^{\lambda\omega(\sigma)}
|\langle u,\Pi(\sigma)\varphi\rangle|\right)^2=
\sup_{\sigma\in\Lambda\cap\Gamma}\left(e^{\lambda\omega(\sigma)}
|\langle u,\Pi(\sigma)\varphi\rangle|\right)^2\leq C_{2\lambda}.
\eeqsn
Therefore \eqref{39Gabor} is proved.
\end{proof}

In the following we need some results on modulation spaces, in particular in the case of exponential weights. We recall here the main definitions and results, referring to \cite{G} for the classical theory and to \cite{BJO-Gabor} for the ultradifferentiable setting.

\begin{Def}
Let $\omega$ be a weight as in Definition \ref{def21Gabor}, and $m_\mu(z):=e^{\mu\omega(z)}$ for $\mu\in\R$ and $z\in\R^{2d}$. We fix a window $\varphi\in\mathcal{S}_\omega(\R^d)\setminus\{0\}$. The modulation space $M^{p,q}_{m_\mu}(\R^d)$, $1\leq p,q\leq +\infty$, is defined as
\beqsn
M^{p,q}_{m_\mu}(\R^d) := \{f\in\mathcal{S}'_\omega(\R^d) : V_\varphi f\in L^{p,q}_{m_\mu}(\R^{2d})\},
\eeqsn
where $L^{p,q}_{m_\mu}(\R^{2d})$ is the usual mixed norm weighted Lebesgue space, defined by
\beqsn
\| F\|_{L^{p,q}_{m_\mu}} :=\left(\int_{\R^d}\left(\int_{\R^d} |F(x,\xi)|^p m_\mu(x,\xi)^p dx\right)^{q/p} d\xi\right)^{1/q} <+\infty
\eeqsn
when $p,q<\infty$, with standard meaning when one of them is $\infty$. We write $M^p_{m_\mu}(\R^d):= M^{p,p}_{m_\mu}(\R^d)$.
\end{Def}
The modulation spaces $M^{p,q}_{m_\mu}(\R^d)$ are Banach spaces with norm $\| f\|_{M^{p,q}_{m_\mu}} :=\| V_\varphi f\|_{L^{p,q}_{m_\mu}}$, and are independent of the window $\varphi$, in the sense that different windows give equivalent norms. Moreover, for $1\leq p,q<\infty$ we have that
\beqs
\label{add-dual-modul}
(M^{p,q}_{m_\mu})^* = M^{p',q'}_{m_{-\mu}},
\eeqs
where $p', q'$ are the conjugate exponents of $p,q$, and
\beqs
\label{S-Sprime-mod}
\mathcal{S}_\omega(\R^d)=\bigcap_{\mu>0} M^{p,q}_{m_\mu}(\R^d)\quad \text{and}\quad \mathcal{S}'_\omega(\R^d)=\bigcup_{\mu<0} M^{p,q}_{m_\mu}(\R^d),
\eeqs
for every $1\leq p,q\leq\infty$ (see also \cite{ABJO-cptWeyl}).

Our aim is to investigate the stability of the Gabor $\omega$-wave front set $\WF^G_\omega(u)$ under perturbations of the Gabor frame $\{x_\sigma\}_{\sigma\in\Lambda}=
\{\Pi(\sigma)\varphi\}_{\sigma\in\Lambda}$ in $L^2(\R^d)$.

\section{Wave front set and $\varepsilon$-perturbations of frames}
\label{sec3}

Given the Hilbert space $L^2(\R^d)$ and a countable set of multi-indices $\Lambda$, we
resume from \cite{C} the notion of $\varepsilon$-perturbation of a frame:

\begin{Def}
\label{defpert}
A family $\{y_\sigma\}_{\sigma\in\Lambda}$ in $L^2(\R^d)$ is an
{\em $\varepsilon$-perturbation} of the frame $\{x_\sigma\}_{\sigma\in\Lambda}$ if
\beqsn
\sum_{\sigma\in\Lambda}\|x_\sigma-y_\sigma\|^2\leq\varepsilon,
\eeqsn
where $\|\cdot\|$ is the norm in $L^2(\R^d)$.
\end{Def}

The following result is essentially \cite[Thm.~1]{C}:

\begin{Lemma}
\label{rem1}
Suppose that $\{x_\sigma\}_{\sigma\in\Lambda}$ is a frame with lower frame bound $A$ and 
upper frame bound $B$, and let $\{y_\sigma\}_{\sigma\in\Lambda}$ be an $\varepsilon$-perturbation
 of $\{x_\sigma\}_{\sigma\in\Lambda}$ . Then:
\begin{itemize}
\item[(i)] For every $\varepsilon>0$,  $\{y_\sigma\}_{\sigma\in\Lambda}$  is a Bessel sequence.
\item[(ii)] If $0<\varepsilon<A$,  then also 
$\{y_\sigma\}_{\sigma\in\Lambda}$ is a frame in $L^2(\R^d)$ with lower frame bound
$A(1-\sqrt{\varepsilon/A})^2$ and upper frame bound
$B(1+\sqrt{\varepsilon/B})^2$.
\end{itemize}
\end{Lemma}

\begin{proof}
Point {\rm (ii)} is exactly \cite[Thm.~1]{C}. Point {\rm (i)} does not require conditions on $\varepsilon$; indeed, for every $f\in L^2(\R^d)$, we have
\beqsn
\sum_{\sigma\in\Lambda} |\langle f,y_\sigma\rangle|^2 &&\leq 2\left( \sum_{\sigma\in\Lambda} |\langle f,y_\sigma-x_\sigma\rangle|^2 + \sum_{\sigma\in\Lambda} |\langle f,x_\sigma\rangle|^2\right) \\
&&\leq 2\left( \| f\|^2 \sum_{\sigma\in\Lambda} \| y_\sigma-x_\sigma\|^2 + B\| f\|^2\right) \\
&&= 2(B+\varepsilon)\| f\|^2.
\eeqsn
\end{proof}

From now on we fix $\varphi\in\mathcal{S}_\omega(\R^d)\setminus\{0\}$ and a lattice 
$\Lambda=\alpha\mathbb{Z}^d\times\beta\mathbb{Z}^d$ such that 
$\{x_\sigma\}_{\sigma\in\Lambda} = \{\Pi(\sigma)\varphi\}_{\sigma\in\Lambda}$ is a Gabor 
frame in $L^2(\R^d)$.

Inspired by \cite{AFGLS}, where stability of phase retrieval by $\varepsilon$-perturbation of
frames is studied, we shall study  here the stability of $\WF^G_\omega(u)$ when
the Gabor frame $\{x_\sigma\}_{\sigma\in\Lambda}$ is replaced by an $\varepsilon$-perturbation 
$\{y_\sigma\}_{\sigma\in\Lambda}$
in \eqref{39Gabor}. 

We start by the following definition:

\begin{Def}
\label{y-sigmaWF}
If $u\in\Sch'_\omega(\R^d)$ and $\{y_\sigma\}_{\sigma\in\Lambda}\subset \mathcal{S}_\omega(\R^d)$,
we say that $z_0\in\R^{2d}\setminus\{0\}$ is not in the 
$\{y_\sigma\}_{\sigma\in\Lambda}$-wave front set
$\WF^{\{y_\sigma\}}_\omega(u)$ of $u$ if there exists an open conic set $\Gamma
\subseteq\R^{2d}\setminus\{0\}$ containing $z_0$ such that
\beqs
\label{3}
\sup_{\sigma\in\Lambda\cap\Gamma}e^{\lambda\omega(\sigma)}
|\langle u,y_\sigma\rangle|<+\infty,
\qquad\forall\lambda>0.
\eeqs
\end{Def}

The notation $\WF^{\{y_\sigma\}}_\omega(u)$ denotes that the 
$\{y_\sigma\}_{\sigma\in\Lambda}$-wave front set depends not only on the weight $\omega$ 
but also on the sequence $\{y_\sigma\}_{\sigma\in\Lambda}$.

\begin{Rem}\label{wave-dual}
In Definition \ref{y-sigmaWF} we supposed that $u$ is a distribution and $y_\sigma$ are in the corresponding test function space; on the other hand, this definition makes sense each time the expression $\langle u,y_\sigma\rangle$ is well defined, for instance when $u$ belongs to a modulation space and $y_\sigma$ are in the corresponding dual space.
\end{Rem}

Similarly as in Lemma~\ref{lemma1}, condition \eqref{3} is equivalent to
\beqs
\label{4}
\sum_{\sigma\in\Lambda\cap\Gamma}e^{\lambda\omega(\sigma)}
|\langle u,y_\sigma\rangle|^2<+\infty,
\qquad\forall\lambda>0.
\eeqs

\begin{Th}
\label{th1}
Let $\varphi\in\Sch_\omega(\R^d)\setminus\{0\}$ and $\Lambda=\alpha\Z^d\times\beta\Z^d$ a lattice with $\alpha,\beta>0$ sufficiently small so that 
$\{x_\sigma\}_{\sigma\in\Lambda}=
\{\Pi(\sigma)\varphi\}_{\sigma\in\Lambda}$ is a Gabor frame in $L^2(\R^d)$.
\begin{itemize}
\item[(i)]
Fix $\mu\geq 0$ and suppose that $\{y_\sigma\}_{\sigma\in\Lambda}\subset M^2_{m_\mu}(\R^d)$ satisfies
\beqs
\label{5}
\forall\lambda\geq0\  \exists \varepsilon_\lambda>0:
\quad \sum_{\sigma\in\Lambda}e^{\lambda\omega(\sigma)}\|x_\sigma-y_\sigma\|_{M^2_{m_\mu}}^2\leq\varepsilon_\lambda;
\eeqs
then, for every $u\in M^2_{m_{-\mu}}(\R^d)$ we have
\beqsn
\WF^G_\omega(u)=\WF^{\{y_\sigma\}}_\omega(u).
\eeqsn
\item[(ii)] suppose that $\{y_\sigma\}_{\sigma\in\Lambda}\subset \Sch_\omega(\R^d)$ satisfies
\beqsn
\forall\lambda,\mu\geq0\  \exists \varepsilon_{\lambda,\mu}>0:
\quad \sum_{\sigma\in\Lambda}e^{\lambda\omega(\sigma)}\|x_\sigma-y_\sigma\|_{M^2_{m_\mu}}^2\leq\varepsilon_{\lambda,\mu};
\eeqsn
then, for every $u\in \Sch_\omega'(\R^d)$ we have
\beqsn
\WF^G_\omega(u)=\WF^{\{y_\sigma\}}_\omega(u).
\eeqsn
\end{itemize}
\end{Th}

\begin{proof}
{\rm (i)} We observe first that the wave front set is well defined in view of Remark \ref{wave-dual} and \eqref{add-dual-modul}. Let us prove that conditions \eqref{1} and \eqref{4} are equivalent, for any fixed open conic set
$\Gamma\subseteq\R^{2d}\setminus\{0\}$.

Indeed, for any $\lambda>0$
\beqs
\nonumber
&&
\sum_{\sigma\in\Lambda\cap\Gamma}e^{\lambda\omega(\sigma)}|\langle u,y_\sigma\rangle|^2-
\sum_{\sigma\in\Lambda\cap\Gamma}e^{\lambda\omega(\sigma)}|\langle u,x_\sigma\rangle|^2\\
\nonumber
=&&
\sum_{\sigma\in\Lambda\cap\Gamma}e^{\lambda\omega(\sigma)}|\langle u,x_\sigma-(x_\sigma-y_\sigma)\rangle|^2-
\sum_{\sigma\in\Lambda\cap\Gamma}e^{\lambda\omega(\sigma)}|\langle u,x_\sigma\rangle|^2\\
\nonumber
\geq&&\sum_{\sigma\in\Lambda\cap\Gamma}e^{\lambda\omega(\sigma)}
(|\langle u,x_\sigma\rangle|-|\langle u,x_\sigma-y_\sigma\rangle|)^2-
\sum_{\sigma\in\Lambda\cap\Gamma}e^{\lambda\omega(\sigma)}|\langle u,x_\sigma\rangle|^2\\
\nonumber
\geq&&-2\sum_{\sigma\in\Lambda\cap\Gamma}|\langle u,x_\sigma\rangle|
e^{\lambda\omega(\sigma)}|\langle u,x_\sigma-y_\sigma\rangle|+
\sum_{\sigma\in\Lambda\cap\Gamma}e^{\lambda\omega(\sigma)}|\langle u,x_\sigma-y_\sigma\rangle|^2\\
\nonumber
\geq&&-2\left(\sum_{\sigma\in\Lambda\cap\Gamma}e^{-2\mu\omega(\sigma)}|\langle u,x_\sigma\rangle|^2\right)^{1/2}
\cdot\left(\sum_{\sigma\in\Lambda\cap\Gamma}e^{2(\lambda+\mu)\omega(\sigma)}|\langle u,x_\sigma-y_\sigma\rangle|^2\right)^{1/2}\\
\label{6}
&&+\sum_{\sigma\in\Lambda\cap\Gamma}e^{\lambda\omega(\sigma)}|\langle u,x_\sigma-y_\sigma\rangle|^2
\eeqs
by the Cauchy-Schwarz inequality in $\ell^2(\R^d)$. Observe that for every $\nu\geq 0$ we have that the series 
$\sum_{\Lambda\cap\Gamma}e^{\nu\omega(\sigma)}|\langle u,x_\sigma-y_\sigma\rangle|^2$ is convergent since
\beqs
\label{add-01}
\sum_{\sigma\in\Lambda\cap\Gamma}e^{\nu\omega(\sigma)}|\langle u,x_\sigma-y_\sigma\rangle|^2\leq 
\sum_{\sigma\in\Lambda}e^{\nu\omega(\sigma)}\| u\|_{M^2_{m_{-\mu}}}^2 \| x_\sigma-y_\sigma\|_{M^2_{m_{\mu}}}^2 \leq \varepsilon_\nu \| u\|_{M^2_{m_{-\mu}}}^2
\eeqs
by condition \eqref{5}. Moreover, from \cite[Theorem 3.13]{BJO-Gabor}, we have continuity of the coefficient operator of the frame $\{x_\sigma\}_{\sigma\in\Lambda}$ on the modulation space $M^2_{m_{-\mu}}(\R^d)$, so there exists $C>0$ such that
\beqs
\label{add-02}
\left(\sum_{\sigma\in\Lambda\cap\Gamma}e^{-2\mu\omega(\sigma)}|\langle u,x_\sigma\rangle|^2\right)^{1/2} \leq
\left(\sum_{\sigma\in\Lambda}e^{-2\mu\omega(\sigma)}|\langle u,x_\sigma\rangle|^2\right)^{1/2}\leq C\| u\|_{M^2_{m_{-\mu}}}.
\eeqs
We can then further estimate, from \eqref{6}, \eqref{add-01} and \eqref{add-02}:
\beqsn
&&
\sum_{\sigma\in\Lambda\cap\Gamma}e^{\lambda\omega(\sigma)}|\langle u,y_\sigma\rangle|^2-
\sum_{\sigma\in\Lambda\cap\Gamma}e^{\lambda\omega(\sigma)}|\langle u,x_\sigma\rangle|^2\geq \\
&&\quad \geq-2C\| u\|_{M^2_{m_{-\mu}}}\sqrt{\varepsilon_{2(\lambda+\mu)}}\| u\|_{M^2_{m_{-\mu}}}
+\sum_{\sigma\in\Lambda\cap\Gamma}e^{\lambda\omega(\sigma)}|\langle u,x_\sigma-y_\sigma\rangle|^2.
\eeqsn
We have thus proved that
\beqsn
\sum_{\sigma\in\Lambda\cap\Gamma}e^{\lambda\omega(\sigma)}|\langle u,x_\sigma\rangle|^2
\leq&& \sum_{\sigma\in\Lambda\cap\Gamma}e^{\lambda\omega(\sigma)}|\langle u,y_\sigma\rangle|^2+
2C\sqrt{\varepsilon_{2(\lambda+\mu)}}\| u\|_{M^2_{m_{-\mu}}}^2\\
&&-
\sum_{\sigma\in\Lambda\cap\Gamma}e^{\lambda\omega(\sigma)}|\langle u,x_\sigma-y_\sigma\rangle|^2\\
\leq&&\sum_{\sigma\in\Lambda\cap\Gamma}e^{\lambda\omega(\sigma)}|\langle u,y_\sigma\rangle|^2+
2C\sqrt{\varepsilon_{2(\lambda+\mu)}}\| u\|_{M^2_{m_{-\mu}}}^2,
\eeqsn
so that \eqref{4} implies \eqref{1}.

For the opposite implication $\eqref{1}\Rightarrow\eqref{4}$, we first remark that, from \eqref{add-01} and \eqref{add-02}
\beqsn
&&\left(\sum_{\sigma\in\Lambda\cap\Gamma}e^{-2\mu\omega(\sigma)}|\langle u,y_\sigma\rangle|^2\right)^{1/2} \leq \\
\leq&&
\sqrt{2}\left[\left(\sum_{\sigma\in\Lambda}e^{-2\mu\omega(\sigma)}|\langle u,y_\sigma-x_\sigma\rangle|^2\right)^{1/2}
+\left(\sum_{\sigma\in\Lambda}e^{-2\mu\omega(\sigma)}|\langle u,x_\sigma\rangle|^2\right)^{1/2}\right] \\
\leq&& \sqrt{2}\left(\sqrt{\varepsilon_0}+C\right) \| u\|_{M^2_{m_{-\mu}}}.
\eeqsn
Hence we can proceed as in the proof of 
$\eqref{4}\Rightarrow\eqref{1}$ by exchanging $x_\sigma$ and $y_\sigma$; from \eqref{6} we get
that
\beqsn
&&
\sum_{\sigma\in\Lambda\cap\Gamma}e^{\lambda\omega(\sigma)}|\langle u,x_\sigma\rangle|^2-
\sum_{\sigma\in\Lambda\cap\Gamma}e^{\lambda\omega(\sigma)}|\langle u,y_\sigma\rangle|^2\\
\geq&&
-2\sqrt{2}\left(\sqrt{\varepsilon_0}+C\right)\sqrt{\varepsilon_{2(\lambda+\mu)}}\| u\|_{M^2_{m_{-\mu}}}^2
+\sum_{\sigma\in\Lambda\cap\Gamma}e^{\lambda\omega(\sigma)}|\langle u,x_\sigma-y_\sigma\rangle|^2.
\eeqsn
Therefore
\beqsn
\sum_{\sigma\in\Lambda\cap\Gamma}e^{\lambda\omega(\sigma)}|\langle u,y_\sigma\rangle|^2
\leq&& \sum_{\sigma\in\Lambda\cap\Gamma}e^{\lambda\omega(\sigma)}|\langle u,x_\sigma\rangle|^2+2\sqrt{2}\left(\sqrt{\varepsilon_0}+C\right)\sqrt{\varepsilon_{2(\lambda+\mu)}}\| u\|_{M^2_{m_{-\mu}}}^2,
\eeqsn
and \eqref{1} implies \eqref{4}.
\vskip0.3cm
\indent {\rm (ii)} Concerning the second point of the theorem, if $u\in \Sch_\omega'(\R^d)$, from \eqref{S-Sprime-mod} for $p=q=2$ we have that there exists $\mu> 0$ such that $u\in M^2_{m_{-\mu}}(\R^d)$, and for that $\mu$ condition \eqref{5} is satisfied, so we can apply point {\rm (i)} and conclude that the two wave front sets coincide.
\end{proof}

\begin{Rem}
We observe that, since $M^2_{m_{\mu}}(\R^d)$ is continuously embedded in $M^2_{m_0}(\R^d)=L^2(\R^d)$ for every $\mu\geq 0$, condition \eqref{5} implies that $\{y_\sigma\}_{\sigma\in\Lambda}$ in Theorem \ref{th1} is an $\varepsilon_0$-perturbation of the frame $\{x_\sigma\}_{\sigma\in\Lambda}$:
\beqsn
\sum_{\sigma\in\Lambda}\|x_\sigma-y_\sigma\|^2\leq \sum_{\sigma\in\Lambda}\|x_\sigma-y_\sigma\|_{M^2_{m_\mu}}^2\leq\varepsilon_0.
\eeqsn
In view of Lemma \ref{rem1},  $\{y_\sigma\}_{\sigma\in\Lambda}$ is a Bessel sequence in $L^2(\R^d)$; moreover, by \cite[Theorem 3]{C}, we have that $\{ y_\sigma\}_{\sigma\in\Lambda}$ is a frame for $\overline{{\rm span}}\{y_\sigma\}_{\sigma\in\Lambda}$ but without additional conditions on $\varepsilon_0$ we don't know if it is a frame for the whole $L^2(\R^d)$, cf. \cite{C}. Theorem \ref{th1} thus gives a characterization of the Gabor wave front set through a sequence of elements that are perturbations of a Gabor frame but are not a Gabor system and do not need even to be a frame. 
\end{Rem}

\section{Nonstationary perturbations of Gabor frames}
\label{sec4}

Given $\varphi\in\Sch_\omega(\R^d)\setminus\{0\}$, our next goal is to study the stability of the Gabor $\omega$-wave front set by
perturbing the Gabor frame $\{\Pi(\sigma)\varphi\}_{\sigma\in\Lambda}
=\{e^{i\beta m\cdot t}\varphi(t-\alpha n)\}_{m,n\in\Z^d}$ with nonstationary frames, where $\beta$
may depend on $n$ or $\alpha$ may depend on $m$.

Consider first the case $\beta=\beta_n$ and set, for $g\in\Sch_\omega(\R^d)\setminus\{0\}$:
\beqsn
&&g_n(t)=g(t-\alpha n),\\
&&g_{m,n}(t)=e^{i\beta m\cdot t}g_n(t)=e^{i\beta_n m\cdot t}g(t-\alpha n)=
\Pi(\alpha n,\beta_n m)g(t).
\eeqsn

Let us start by a preliminary result, that is taken from \cite{Balazs}, adapted to our case:
\begin{Prop}
\label{lemma2}
Let $g\in\Sch_\omega(\R^d)\setminus\{0\}$ with
$\supp g\subseteq[A_1,B_1]\times\cdots\times[A_d,B_d]$ and
\beqs
\label{7}
\frac{1}{\beta_n}\geq\max_{1\leq j\leq d}(B_j-A_j),\quad\forall n\in\Z^d.
\eeqs
Then $\{g_{m,n}\}_{m,n\in\Z^d}$ is a frame in $L^2(\R^d)$, with lower and upper frame bounds $A$ and $B$ respectively, if and only if
\beqs
\label{10}
&&\inf_{t\in\R^d}\sum_{n\in\Z^d}\frac{1}{\beta_n^d}|g_n(t)|^2=A>0\\
\label{11}
&&\sup_{t\in\R^d}\sum_{n\in\Z^d}\frac{1}{\beta_n^d}|g_n(t)|^2=B<+\infty.
\eeqs
In this case, the frame operator
\beqsn
S:\ L^2(\R^d)&&\longrightarrow L^2(\R^d)\\
f&&\longmapsto \sum_{m,n\in\Z^d}\langle f, g_{m,n}\rangle g_{m,n}
\eeqsn
is given by the multiplication operator
\beqs
\label{17}
Sf=\sum_{n\in\Z^d}\frac{1}{\beta_n^d}|g_n|^2f
\eeqs
and the canonical dual frame is given by
\beqs
\label{14}
\tilde{g}_{m,n}(t)=\frac{g_n(t)}{ \sum_{\ell\in\Z^d}\frac{1}{\beta_\ell^d}|g_\ell(t)|^2}
e^{i\beta_n m\cdot t}.
\eeqs
\end{Prop}

\begin{proof}
Let us first remark that $\{\beta_n\}_{n\in\Z^d}$ is not only uniformly bounded from above by \eqref{7}, but it is also bounded from below from assumption \eqref{11}, since
\beqsn
\frac{1}{\beta_n^d}|g(t-\alpha n)|^2\leq\sum_{n\in\Z^d}\frac{1}{\beta_n^d}
|g(t-\alpha n)|^2\leq B\qquad\forall t\in\R^d
\eeqsn
implies 
\beqsn
\frac{1}{\beta_n^d}\sup_{t\in\R^d}|g(t-\alpha n)|^2\leq B
\eeqsn
and hence
\beqs
\label{12}
\beta_n^d\geq\frac BL>0
\eeqs
for
\beqsn
L:=\sup_{t\in\R^d}|g(t)|^2=\sup_{t\in\R^d}|g(t-\alpha n)|^2.
\eeqsn

Now we consider the series
\beqsn
\sum_{m,n\in\Z^d}|\langle f,g_{m,n}\rangle|^2=\sum_{m,n\in\Z^d}\left|\int_{\R^d}f(t)\bar{g}_n(t)
e^{-i\beta_n m\cdot t}dt\right|^2;
\eeqsn
by assumption \eqref{7}
\beqsn
\supp g_n\subseteq&&\prod_{j=1}^d[A_j+\alpha n_j,B_j+\alpha n_j]=:J_n\\
\subseteq&&\prod_{j=1}^d[A_j+\alpha n_j,A_j+\alpha n_j+\frac{1}{\beta_n}]=:I_n
\eeqsn
and hence
\beqs
\nonumber
\sum_{m,n\in\Z^d}|\langle f,g_{m,n}\rangle|^2=&&\sum_{m,n\in\Z^d}\left|\int_{I_n}f(t)\bar{g}_n(t)e^{-i\beta_n m\cdot t}dt\right|^2\\
\label{8}
=&&\sum_{n\in\Z^d}\sum_{m\in\Z^d}|\langle f\bar{g}_n,e^{i\beta_n m\cdot t}\rangle|_{L^2(I_n)}^2.
\eeqs

Noting that $\sum_{\ell\in\Z^d}T_{\frac{\ell}{\beta_n}}(f\bar{g}_n)(t)$ is a $1/\beta_n$-periodic 
function and writing
\beqsn
f\bar{g}_n=\chi_{I_n}\sum_{\ell\in\Z^d}T_{\frac{\ell}{\beta_n}}(f\bar{g}_n),
\eeqsn
where $\chi_{I_n}$ is the characteristic function of $I_n$, we can interprete $\langle f\bar{g}_n,e^{i\beta_n m\cdot t}\rangle_{L^2(I_n)}$ as the Fourier coefficients of
$f\bar{g}_n$ and apply Parseval's identity in \eqref{8} to obtain:
\beqs
\nonumber
\sum_{m,n\in\Z^d}|\langle f,g_{m,n}\rangle|^2=&&\sum_{n\in\Z^d}\frac{1}{\beta_n^d}\|f\bar{g}_n\|^2_{L^2(I_n)}
=\sum_{n\in\Z^d}\frac{1}{\beta_n^d}\langle f\bar{g}_n,f\bar{g}_n\rangle_{L^2(I_n)}\\
\label{9}
=&&\sum_{n\in\Z^d}\frac{1}{\beta_n^d}\langle f|{g}_n|^2,f\rangle_{L^2(I_n)}
=\Big\langle\sum_{n\in\Z^d}\frac{1}{\beta_n^d}|g_n|^2f,f\Big\rangle_{L^2(\R^d)}
\eeqs
by Lebesgue's dominated convergence theorem, since for all $N\in\N$
\beqsn
\sum_{\afrac{n\in\Z^d}{|n|\leq N}}\frac{1}{\beta_n^d}|g(t-\alpha n)|^2|f(t)|^2
\leq \sum_{n\in\Z^d}\frac{1}{\beta_n^d}|g_n(t)|^2|f(t)|^2
\leq B|f(t)|^2\in L^1(\R^d)
\eeqsn
by assumption \eqref{11}. In particular, the series $\sum_{m,n\in\Z^d}|\langle f,g_{m,n}\rangle|^2$ is convergent for every $f\in L^2(\R^d)$, which means that $\{g_{m,n}\}_{m,n\in\Z^d}$ is a Bessel sequence in $L^2$; we then have that the series 
$\sum_{m,n\in\Z^d}\langle f,g_{m,n}\rangle g_{m,n}$ converges unconditionally, and from \eqref{9}
\beqs
\label{add-03}
\langle Sf,f\rangle =&&\sum_{m,n\in\Z^d}\langle f,g_{m,n}\rangle \langle g_{m,n},f\rangle =\sum_{m,n\in\Z^d}|\langle f,g_{m,n}\rangle|^2 =\Big\langle\sum_{n\in\Z^d}\frac{1}{\beta_n^d}|g_n|^2f,f\Big\rangle.
\eeqs
Since the frame operator is positive and self-adjoint (see \cite[\S~5.1]{G}), from \cite[Thm.~7.16]{A} and \eqref{add-03} we have that the frame operator $S$ is the multiplication operator \eqref{17}.
Moreover, it follows from \eqref{add-03} that $\{g_{m,n}\}_{m,n\in\Z^d}$ is a frame in $L^2(\R^d)$ if and only if
conditions \eqref{10} and \eqref{11} are satisfied and in this case the canonical dual window
$\tilde{g}_{m,n}$ of $g_{m,n}$ is given by
\beqsn
\tilde{g}_{m,n}(t)=S^{-1}g_{m,n}(t)=\frac{g_n(t)}{ \sum_{\ell\in\Z^d}\frac{1}{\beta_\ell^d}
|g_\ell(t)|^2}e^{i\beta_n m\cdot t}.
\eeqsn
\end{proof}

Note that, even if $\{g_{m,n}\}_{m,n\in\Z^d}$ is a frame, if the sequence $\{\beta_n\}$ is not constant, in general $\{g_{m,n}\}$ is not an $\varepsilon$-perturbation (according to Definition~\ref{defpert})
of the Gabor frame $\{\varphi_{m,n}\}:=\{\Pi(\alpha n,\beta m)g\}$, for any $\beta\in\R$, and so we cannot directly apply Theorem~\ref{th1}
to show that $\WF^{\{g_{m,n}\}}_\omega(u)$ is equal to $\WF^G_\omega(u)$, for $u\in\Sch'_\omega(\R^d)$.
Indeed, taking for example $g(t)=e^{-t^2/2}$, for $t\in\R$, we can compute:
\beqsn
\|\varphi_{m,n}-g_{m,n}\|^2=&&
\|e^{i\beta m t}g(t-\alpha n)-e^{i\beta_n m t}g(t-\alpha n)\|^2\\
=&&\int_{-\infty}^{+\infty}|e^{i\beta mt}-e^{i\beta_n mt}|^2e^{-(t-\alpha n)^2}dt\\
=&&\int_{-\infty}^{+\infty}|1-e^{i(\beta_n-\beta) m t}|^2e^{-(t-\alpha n)^2}dt\\
=&&\int_{-\infty}^{+\infty}[(1-\cos((\beta_n-\beta)mt))^2+\sin^2((\beta_n-\beta)mt)]e^{-(t-\alpha n)^2}dt.
\eeqsn
By the change of variables $s=t-\alpha n$ we have
\beqs
\nonumber
\|\varphi_{m,n}-g_{m,n}\|^2=&&
\int_{-\infty}^{+\infty}(2-2\cos((\beta_n-\beta)m(s+\alpha n))e^{-s^2}ds\\
\nonumber
=&&2\int_{-\infty}^{+\infty}e^{-s^2}ds-2\int_{-\infty}^{+\infty}\cos((\beta_n-\beta)ms)
\cos((\beta_n-\beta)m\alpha n)e^{-s^2}ds\\
\nonumber
&&+2\int_{-\infty}^{+\infty}\sin((\beta_n-\beta)ms)
\sin((\beta_n-\beta)m\alpha n)e^{-s^2}ds\\
\nonumber
=&&2\sqrt{\pi}-2\cos((\beta_n-\beta)m\alpha n)\int_{-\infty}^{+\infty}\cos((\beta_n-\beta)ms)
e^{-s^2}ds\\
\nonumber
&&+2\sin((\beta_n-\beta)m\alpha n)\int_{-\infty}^{+\infty}\sin((\beta_n-\beta)ms)
e^{-s^2}ds\\
\label{27''}
=&&2\sqrt{\pi}-2\cos((\beta_n-\beta)m\alpha n)\sqrt{\pi}e^{-\frac{(\beta_n-\beta)^2m^2}{4}}
\eeqs
since $\int_{-\infty}^{+\infty}e^{-s^2}ds=\sqrt{\pi}$, $\int_{-\infty}^{+\infty}\sin((\beta_n-\beta)ms)
e^{-s^2}ds=0$ being the integral of an odd function, and by the holomorphy of $e^{-u^2}$:
\beqsn
\int_{-\infty}^{+\infty}\cos(ax)e^{-x^2}dx=&&\Re \int_{-\infty}^{+\infty}e^{iax}e^{-x^2}dx
=\Re \int_{-\infty}^{+\infty}e^{-\left(x-i\frac{a}{2}\right)^2-\frac{a^2}{4}}dx\\
=&&e^{-\frac{a^2}{4}}\Re\int_{-\infty-i\frac a2}^{+\infty-i\frac a2}e^{-u^2}du
=e^{-\frac{a^2}{4}}\Re\int_{-\infty}^{+\infty}e^{-u^2}du\\
=&&\sqrt{\pi}e^{-\frac{a^2}{4}}.
\eeqsn

From \eqref{27''} we have that
\beqsn
\|\varphi_{m,n}-g_{m,n}\|^2=
2\sqrt{\pi}\left(1-\cos((\beta_n-\beta)m\alpha n)e^{-\frac{(\beta_n-\beta)^2m^2}{4}}\right)
\eeqsn
is an even function of $m$ that tends to $2\sqrt{\pi}$ as $m\to+\infty$ for every fixed $n$ such that $\beta_n\neq\beta$; then, if $\{\beta_n\}$ is not constant, for every choice of $\beta$ we have 
\beqsn
\sum_{m,n\in\Z}\|\varphi_{m,n}-g_{m,n}\|^2=+\infty
\eeqsn
and $\{g_{m,n}\}_{m,n\in\Z^d}$
is not an $\varepsilon$-perturbation of 
$\{\varphi_{m,n}\}_{m,n\in\Z^d}$. Of course one could wonder if there exists another Gabor frame (with another window and another lattice) of which $\{g_{m,n}\}_{m,n\in\Z^d}$ is an $\varepsilon$-perturbation, but this seems a difficult problem in general. 

We thus need a different approach to study the wave front set $\WF^{\{g_{m,n}\}}_\omega(u)$.
First remark that $g_{m,n}\in\Sch_\omega(\R^d)$ since $g\in\Sch_\omega(\R^d)$.
We shall need in the following that also $\tilde{g}_{m,n}\in\Sch_\omega(\R^d)$.
To this aim we prove that
\beqs
\label{30}
\frac{1}{G(t)}:=\frac{1}{ \sum_{\ell\in\Z^d}\frac{1}{\beta_\ell^d}
|g_\ell(t)|^2}
\eeqs
is a {\em multiplier} in $\Sch_\omega(\R^d)$, i.e. $\frac 1G f\in\Sch_\omega(\R^d)$ for all
$f\in\Sch_\omega(\R^d)$.

Let us recall from \cite{AM-multipliers} (cf. also \cite{ABJO-cptWeyl} for equivalent formulations),
that a $C^\infty$ function $F$ is a multiplier in $\Sch_\omega(\R^d)$ if and only if $F$ is in the space 
\beqsn
\calO_{M,\omega}(\R^d):=\{F\in C^\infty(\R^d):&&
\forall\lambda>0\, \exists C_\lambda,\mu_\lambda>0\, \mbox{s.t.}\\
&&|D^\kappa F(t)|\leq C_\lambda e^{\lambda\varphi^*\left(\frac{|\kappa|}{\lambda}\right)}
e^{\mu_\lambda \omega(t)}\ \forall\kappa\in\N_0^d,t\in\R^d\}.
\eeqsn

\begin{Prop}
\label{lemma3}
Under the assumptions of Proposition~\ref{lemma2}, the function $\frac{1}{G(t)}$ defined
in \eqref{30} is a multiplier in $\Sch_\omega(\R^d)$, and hence the canonical dual frame 
$\tilde{g}_{m,n}$ defined in \eqref{14} is in $\Sch_\omega(\R^d)$.
\end{Prop}

\begin{proof}
Let us prove that $\frac{1}{G(t)}\in\calO_{M,\omega}(\R^d)$. For all $\lambda>0$ we must find
$C_\lambda,\mu_\lambda>0$ such that
\beqs
\label{18}
\left|D^\kappa\frac{1}{G(t)}\right|
\leq C_\lambda e^{\lambda\varphi^*\left(\frac{|\kappa|}{\lambda}\right)}
e^{\mu_\lambda \omega(t)}\quad\forall\kappa\in\N_0^d,t\in\R^d.
\eeqs

By the Fa\`a di Bruno formula (see, for instance, \cite{BM}):
\beqs
\label{21'}
D^\kappa\frac{1}{G(t)}
=&&\sum_{1\leq\ell\leq|\kappa|}\kappa!\left.\left(\frac 1x\right)^{(\ell)}\right|_{x=G(t)}
\sum_{\substack{c_\gamma\in\N_0\\\sum_{|\gamma|>0}c_\gamma=\ell\\\sum_{|\gamma|>0}\gamma c_\gamma=\kappa}}\prod_{|\gamma|>0}\frac{1}{c_\gamma!}
\left(\frac{D^\gamma G(t)}{\gamma!}\right)^{c_\gamma}\\
\label{21}
=&&\sum_{1\leq\ell\leq|\kappa|}\kappa!\frac{(-1)^\ell \ell!}{(G(t))^{\ell+1}}
\sum_{\substack{c_\gamma\in\N_0\\\sum_{|\gamma|>0}c_\gamma=\ell\\\sum_{|\gamma|>0}\gamma c_\gamma=\kappa}}\prod_{|\gamma|>0}\frac{1}{c_\gamma!}
\left(\frac{D^\gamma G(t)}{\gamma!}\right)^{c_\gamma}.
\eeqs

To estimate $D^\gamma G(t)$ we shall only need that $g\in\Sch_\omega(\R^d)$ has compact support and $\{\beta_n\}_{n\in\Z^d}$ is bounded (because of \eqref{7} and \eqref{12}).
For this reason, without loss of generality, we can estimate, instead of $D^\gamma G$, the $\gamma$-derivatives of
\beqs
\label{22}
\tilde{G}(t)=\sum_{n\in\Z^d}g(t-\alpha n)
\eeqs
for a real valued non-negative $g\in\Sch_\omega(\R^d)$ with compact support
(we don't mind about the square since $\Sch_\omega(\R^d)$ is an algebra; see \cite{Fieker,Bjorck}).
Remark also that $\tilde{G}(t)$ is $\alpha$-periodic and the sum in \eqref{22} is locally finite since $g$ has compact support.
We can then derive term by term and estimate, for some $\bar N\in\N$ depending only on the support of $g$ and for all $\lambda>0$:
\beqs
\nonumber
|D^\gamma\tilde G(t)|\leq &&\sum_{n\in\Z^d}|D^\gamma g(t-\alpha n)|
\leq \sup_{t\in[0,\alpha]}\sum_{n\in\Z^d}|D^\gamma g(t-\alpha n)|\\
\nonumber
=&&\sup_{t\in[0,\alpha]}\sum_{\afrac{n\in\Z^d}{|n|\leq\bar N}}|D^\gamma g(t-\alpha n)|
\leq\sup_{t\in[0,\alpha]}\sum_{\afrac{n\in\Z^d}{|n|\leq\bar N}}
C_\lambda e^{\lambda\varphi^*\left(\frac{|\gamma|}{\lambda}\right)}\\
\label{23}
\leq&&C_{\lambda,\bar N} e^{\lambda\varphi^*\left(\frac{|\gamma|}{\lambda}\right)}
\eeqs
for some $C_\lambda,C_{\lambda,\bar N}>0$, since $g\in\Sch_\omega(\R^d)$.

Substituting \eqref{23} into \eqref{21} and applying \eqref{10}, we get:
\beqs
\nonumber
\left|D^\kappa\frac{1}{G(t)}\right|
\leq&&\sum_{1\leq\ell\leq|\kappa|}\kappa!\left|\frac{(-1)^\ell \ell!}{(G(t))^{\ell+1}}\right|
\sum_{\substack{c_\gamma\in\N_0\\\sum_{|\gamma|>0}c_\gamma=\ell\\\sum_{|\gamma|>0}\gamma c_\gamma=\kappa}}\prod_{|\gamma|>0}\frac{1}{c_\gamma!}
\left|\frac{D^\gamma G(t)}{\gamma!}\right|^{c_\gamma}\\
\label{27}
\leq&&\sum_{1\leq\ell\leq|\kappa|}\frac{\kappa! \ell!}{A^{\ell+1}}
\sum_{\substack{c_\gamma\in\N_0\\\sum_{|\gamma|>0}c_\gamma=\ell\\\sum_{|\gamma|>0}\gamma c_\gamma=\kappa}}\prod_{|\gamma|>0}\frac{1}{c_\gamma!}
\left(\frac{C_{2^d\lambda,\bar N}e^{2^d\lambda\varphi^*\left(\frac{|\gamma|}{2^d\lambda}\right)}}{\gamma!}\right)^{c_\gamma}.
\eeqs

Inspired by the techniques in \cite[Prop.~2.1]{FG}, we use the convexity of $\varphi^*$ to use  repeatedly the following estimate (see also \cite[Lemma~A.1(ix)]{BJO-realPW}):
\beqs
\label{FG1}
e^{2^d\lambda\varphi^*\left(\frac{\gamma_1+\dots+\gamma_d}{2^d\lambda}\right)}\leq&&
D_\lambda
e^{2^{d-1}\lambda\varphi^*\left(\frac{\gamma_1+\dots+\gamma_d-1}{2^{d-1}\lambda}\right)}
\eeqs
for some $D_\lambda>0$.
To adapt this technique to the multidimensional case we consider, for any fixed $\gamma\in\N_0^d$, the multiindex $\delta^{(\gamma)}\in\N_0^d$
whose $j$-th component is defined by
\beqsn
\delta^{(\gamma)}_j=\begin{cases}
1, & \mbox{if } \gamma_j>0\cr 
0, & \mbox{if } \gamma_j=0.
\end{cases}
\eeqsn
Note that $|\delta^{(\gamma)}|\leq d$. Then, using again \eqref{FG1}, we can write:
\beqs
\nonumber
e^{2^d\lambda\varphi^*\left(\frac{\gamma_1+\dots+\gamma_d}{2^d\lambda}\right)}\leq&&
D'_\lambda
e^{2^{d-2}\lambda\varphi^*\left(\frac{\gamma_1+\dots+\gamma_d-2}{2^{d-2}\lambda}\right)}
\leq\dots\\ 
\label{24}
\leq&&
D_{\lambda,d}
e^{\lambda\varphi^*\left(\frac{\gamma_1+\dots+\gamma_d-|\delta^{(\gamma)}|}{\lambda}\right)}
=D_{\lambda,d}
e^{\lambda\varphi^*\left(\frac{|\gamma-\delta^{(\gamma)}|}{\lambda}\right)}
\eeqs
for some $D'_\lambda,D_{\lambda,d}>0$.

Moreover, by the subadditivity of $\omega$,
\beqs
\nonumber
\frac{e^{\lambda\varphi^*\left(\frac{|\gamma|}{\lambda}\right)}}{\gamma!}
\cdot
\frac{e^{\lambda\varphi^*\left(\frac{|\sigma|}{\lambda}\right)}}{\sigma!}=&&
\sup_{s\geq0}\frac{e^{|\gamma|s-\lambda\varphi(s)}}{\gamma!}\cdot
\sup_{t\geq0}\frac{e^{|\sigma|t-\lambda\varphi(t)}}{\sigma!}
=\sup_{s',t'\geq1}\frac{e^{|\gamma|\log s'+|\sigma|\log t'-\lambda(\omega(s')+\omega(t'))}}{\gamma!\sigma!}\\
\nonumber
\leq&&\sup_{s',t'\geq1}\frac{(s')^{|\gamma|}(t')^{|\sigma|}}{\gamma!\sigma!}
e^{-\lambda\omega(s'+t')}
=\sup_{s',t'\geq1}\prod_{j=1}^d
\frac{(s')^{\gamma_j}(t')^{\sigma_j}}{\gamma_j!\sigma_j!}
e^{-\lambda\omega(s'+t')}\\
\nonumber
\leq&&\sup_{s',t'\geq1}\prod_{j=1}^d
\frac{(s'+t')^{\gamma_j+\sigma_j}}{(\gamma_j+\sigma_j)!}
e^{-\lambda\omega(s'+t')}
=\sup_{s',t'\geq1}
\frac{(s'+t')^{|\gamma+\sigma|}}{(\gamma+\sigma)!}
e^{-\lambda\omega(s'+t')}\\
\nonumber
=&&\frac{1}{(\gamma+\sigma)!}\sup_{s',t'\geq1}
e^{|\gamma+\sigma|\log(s'+t')-\lambda\omega(s'+t')}
\leq\frac{1}{(\gamma+\sigma)!}\sup_{x\geq0}
e^{|\gamma+\sigma|x-\lambda\varphi(x)}\\
\label{25}
=&&\frac{e^{\lambda\varphi^*\left(\frac{|\gamma+\sigma|}{\lambda}\right)}}{(\gamma+\sigma)!}.
\eeqs

Then, by repeatedly applying \eqref{25}:
\beqs
\nonumber
\prod_{|\gamma|>0}\Bigg(
\frac{e^{\lambda\varphi^*\left(\frac{|\gamma-\delta^{(\gamma)}|}{\lambda}\right)}}{\gamma!}\Bigg)^{c_\gamma}\leq&&
\prod_{|\gamma|>0}\Bigg(
\frac{e^{\lambda\varphi^*\left(\frac{|\gamma-\delta^{(\gamma)}|}{\lambda}\right)}}{(\gamma-\delta^{(\gamma)})!}\Bigg)^{c_\gamma}
\leq\prod_{|\gamma|>0}\Bigg(
\frac{e^{\lambda\varphi^*\left(\frac{|\gamma-\delta^{(\gamma)}|c_\gamma}{\lambda}\right)}}
{((\gamma-\delta^{(\gamma)}) c_\gamma)!}
\Bigg)\\
\label{26}
\leq&&
\frac{e^{\lambda\varphi^*\left(\frac{\sum_{|\gamma|>0}|\gamma-\delta^{(\gamma)}|c_\gamma}{\lambda}\right)}}{\left(\sum_{|\gamma|>0}(\gamma-\delta^{(\gamma)}) c_\gamma\right)!}
\leq 
\frac{e^{\lambda\varphi^*\left(\frac{|\kappa|-\sum_{|\gamma|>0}|\delta^{(\gamma)}|c_\gamma}{\lambda}\right)}}{(\kappa-\sum_{|\gamma|>0}\delta^{(\gamma)}c_\gamma)!}\,,
\eeqs
since, for all $1\leq j\leq d$,
\beqsn
\sum_{|\gamma|>0}(\gamma_j-\delta^{(\gamma)}_j)c_\gamma
=\sum_{|\gamma|>0}\gamma_jc_\gamma-\sum_{|\gamma|>0}\delta^{(\gamma)}_jc_\gamma
=\kappa_j-\sum_{|\gamma|>0}\delta^{(\gamma)}_jc_\gamma
\eeqsn
and
\beqsn
\sum_{|\gamma|>0}|\gamma-\delta^{(\gamma)}|c_\gamma
=\sum_{j=1}^d(\kappa_j-\sum_{|\gamma|>0}\delta^{(\gamma)}_jc_\gamma)=|\kappa|-
\sum_{|\gamma|>0}|\delta^{(\gamma)}|c_\gamma.
\eeqsn

Substituting \eqref{24} and then \eqref{26} into \eqref{27} we obtain that
\beqs
\nonumber
\left|D^\kappa\frac{1}{G(t)}\right|
\leq&&\sum_{1\leq\ell\leq|\kappa|}\ell!\sum_{\substack{c_\gamma\in\N_0\\\sum_{|\gamma|>0}c_\gamma=\ell\\\sum_{|\gamma|>0}\gamma c_\gamma=\kappa}}\frac{(\sum_{|\gamma|>0}\delta^{(\gamma)}c_\gamma)!}{A^{\ell+1}} \,\frac{\kappa!}{(\sum_{|\gamma|>0}\delta^{(\gamma)}c_\gamma)!(\kappa-\sum_{|\gamma|>0}\delta^{(\gamma)}c_\gamma)!}\\
\nonumber
&&\cdot\prod_{|\gamma|>0}\frac{1}{c_\gamma!}
(C_{2^d\lambda,\bar N}D_{\lambda,d})^{c_\gamma}
e^{\lambda\varphi^*\left(\frac{|\kappa|-\sum_{|\gamma|>0}|\delta^{(\gamma)}|c_\gamma}{\lambda}\right)}\\
\label{A27}
\leq&&\sum_{1\leq\ell\leq|\kappa|}
\ell!\!\!\!
\sum_{\substack{c_\gamma\in\N_0\\\sum_{|\gamma|>0}c_\gamma=\ell\\\sum_{|\gamma|>0}\gamma c_\gamma=\kappa}}(\sum_{|\gamma|>0}\delta^{(\gamma)}c_\gamma)!A_\lambda^\ell 2^{|\kappa|}
e^{\lambda\varphi^*\left(\frac{|\kappa|-\sum_{|\gamma|>0}|\delta^{(\gamma)}|c_\gamma}{\lambda}\right)}\prod_{|\gamma|>0}\frac{1}{c_\gamma!}
\eeqs
for some $A_\lambda\geq1$, being
\beqsn
&&\prod_{|\gamma|>0}(C_{2^d\lambda,\bar N}D_{\lambda,d})^{c_\gamma}
=(C_{2^d\lambda,\bar N}D_{\lambda,d})^{\sum_{|\gamma|>0}c_\gamma}
=(C_{2^d\lambda,\bar N}D_{\lambda,d})^\ell.
\eeqsn
Now, by \cite[Lemma~A.1 (viii), (ii)]{BJO-realPW} we have that
\beqs
\nonumber
&&(\sum_{|\gamma|>0}\delta^{(\gamma)}c_\gamma)!A_\lambda^\ell 
e^{\lambda\varphi^*\left(\frac{|\kappa|-\sum_{|\gamma|>0}|\delta^{(\gamma)}|c_\gamma}{\lambda}\right)} \leq\\
\nonumber &&\qquad\leq(\sum_{|\gamma|>0}\delta^{(\gamma)}c_\gamma)!\,A_\lambda^{(\sum_{|\gamma|>0}|\delta^{(\gamma)}|c_\gamma)} 
e^{\lambda\varphi^*\left(\frac{|\kappa|-\sum_{|\gamma|>0}|\delta^{(\gamma)}|c_\gamma}{\lambda}\right)} \\
\nonumber &&\qquad\leq A'_\lambda\, e^{\lambda\varphi^*\left(\frac{\sum_{|\gamma|>0}|\delta^{(\gamma)}|c_\gamma}{\lambda}\right)} e^{\lambda\varphi^*\left(\frac{|\kappa|-\sum_{|\gamma|>0}|\delta^{(\gamma)}|c_\gamma}{\lambda}\right)} \\
\label{R3-1} &&\qquad\leq A'_\lambda\, e^{\lambda\varphi^*\left(\frac{|\kappa|}{\lambda}\right)}
\eeqs
for some $A'_\lambda>0$, since
$\ell=\sum_{|\gamma|>0}c_\gamma\leq\sum_{|\gamma|>0}|\delta^{(\gamma)}|c_\gamma$, being
$|\gamma|=0$ if $|\delta^{(\gamma)}|=0$.

Substituting \eqref{R3-1} in \eqref{A27} we then have
\beqs
\label{R3-2}
\left|D^\kappa\frac{1}{G(t)}\right|
\leq\sum_{1\leq\ell\leq|\kappa|}2^{|\kappa|} A'_\lambda\, e^{\lambda\varphi^*\left(\frac{|\kappa|}{\lambda}\right)}
\ell!\sum_{\substack{c_\gamma\in\N_0\\\sum_{|\gamma|>0}c_\gamma=\ell\\\sum_{|\gamma|>0}\gamma c_\gamma=\kappa}}\prod_{|\gamma|>0}\frac{1}{c_\gamma!}.
\eeqs

We now claim that
\beqs
\label{combinatoria}
\ell!
\sum_{\substack{c_\gamma\in\N_0\\\sum_{|\gamma|>0}c_\gamma=\ell\\\sum_{|\gamma|>0}\gamma c_\gamma=\kappa}}\prod_{|\gamma|>0}\frac{1}{c_\gamma!}
\leq 2^{|\kappa|+d\ell-d}.
\eeqs

To prove \eqref{combinatoria} we set
\beqsn
J_\kappa:=\{\gamma\in\N_0^d: |\gamma|>0,\gamma\leq\kappa\}
\eeqsn
and consider, for any given $A_\gamma>0$, $\gamma\in\N_0^d$, the following sums of products
\beqsn
\Sigma:=&&\sum_{\substack{\gamma_1+\dots+\gamma_\ell=\kappa\\|\gamma_i|>0}}
A_{\gamma_1}\dots A_{\gamma_\ell}\\
S:=&&\sum_{\gamma\in J_\kappa}A_\gamma.
\eeqsn
Clearly 
\beqsn
S^\ell=\sum_{\gamma_1,\dots,\gamma_\ell\in J_\kappa}A_{\gamma_1}\dots A_{\gamma_\ell}
\eeqsn
contains all the factors that appear in $\Sigma$ and by the multinomial formula
\beqsn
S^\ell=\sum_{\substack{c_\gamma\in\N_0\\\sum_{\gamma\in J_\kappa}c_\gamma=\ell}}
\frac{\ell!}{\prod_{\gamma\in J_\kappa}c_\gamma!}
\prod_{\gamma\in J_\kappa}A_\gamma^{c_\gamma}.
\eeqsn
In order to restrict the terms of $S^\ell$ to the only factors that appear in $\Sigma$ we must
impose the condition
\beqs
\label{gammacgamma}
\sum_{\gamma\in J_\kappa}\gamma c_\gamma=\gamma_1+\dots+\gamma_\ell=\kappa.
\eeqs
Therefore
\beqs
\label{sigma2}
\Sigma=\sum_{\substack{c_\gamma\in\N_0\\\sum_{\gamma\in J_\kappa}c_\gamma=\ell
\\\sum_{\gamma\in J_\kappa}\gamma c_\gamma=\kappa}}
\frac{\ell!}{\prod_{\gamma\in J_\kappa}c_\gamma!}
\prod_{\gamma\in J_\kappa}A_\gamma^{c_\gamma}.
\eeqs
Since condition \eqref{gammacgamma} implies $\kappa\geq\gamma$, substituing 
\eqref{sigma2} in the definition of $\Sigma$ we can write
\beqs
\label{A28}
\ell!
\sum_{\substack{c_\gamma\in\N_0\\\sum_{|\gamma|>0}c_\gamma=\ell\\\sum_{|\gamma|>0}\gamma c_\gamma=\kappa}}\prod_{|\gamma|>0}\frac{1}{c_\gamma!}A_\gamma^{c_\gamma}
=\sum_{\substack{\gamma_1+\dots+\gamma_\ell=\kappa\\|\gamma_i|>0}}
A_{\gamma_1}\dots A_{\gamma_\ell}.
\eeqs
Choosing $A_\gamma=1$ for every $\gamma$ in \eqref{A28} we finally have that
\beqs
\nonumber
&&\ell!\sum_{\substack{c_\gamma\in\N_0\\\sum_{|\gamma|>0}c_\gamma=\ell\\\sum_{|\gamma|>0}\gamma c_\gamma=\kappa}}\prod_{|\gamma|>0}\frac{1}{c_\gamma!}
=\sum_{\substack{\gamma_1+\dots+\gamma_\ell=\kappa\\|\gamma_i|>0}}1\\
\nonumber
\leq&&\#\{(\gamma_1,\dots,\gamma_\ell): \gamma_i\in\N_0^d,\gamma_1+\dots+\gamma_\ell=
\kappa\}\\
\nonumber
=&&\prod_{j=1}^d
\#\{((\gamma_1)_j,\dots,(\gamma_\ell)_j)\in\N_0^\ell: (\gamma_1)_j+\dots+(\gamma_\ell)_j=
\kappa_j\}\\
\label{multinomialcoefficients}
=&&\prod_{j=1}^d
\binom{\kappa_j+\ell-1}{\ell-1}
\leq\prod_{j=1}^d2^{\kappa_j+\ell-1}=2^{|\kappa|+d\ell-d},
\eeqs
by using in \eqref{multinomialcoefficients}
the formula for the 
number of $\ell$-tuples of non-negative integers whose sum is $\kappa_j$.
Therefore \eqref{combinatoria} is proved.

From \eqref{R3-2} and \eqref{combinatoria} we get
\beqsn
\left|D^\kappa\frac{1}{G(t)}\right|
\leq 2^{|\kappa|} A'_\lambda\, e^{\lambda\varphi^*\left(\frac{|\kappa|}{\lambda}\right)} \sum_{1\leq\ell\leq|\kappa|}
2^{|\kappa|+d\ell-d}\leq (8\cdot 2^d)^{|\kappa|} 2^{-d} A'_\lambda \, e^{\lambda\varphi^*\left(\frac{|\kappa|}{\lambda}\right)}.
\eeqsn
Finally, applying \cite[Lemma~A.1 (iv)]{BJO-realPW} we obtain that for every $\lambda>0$
there exists $C_\lambda>0$ such that
\beqsn
\left|D^\kappa\frac{1}{G(t)}\right|
\leq C_\lambda e^{\lambda\varphi^*\left(\frac{|\kappa|}{\lambda}\right)}.
\eeqsn

This proves \eqref{18}, hence $1/G\in\calO_{M,\omega}(\R^d)$ and $\tilde{g}_{m,n}\in\Sch_\omega(\R^d)$.
\end{proof}

We shall also need the following
\begin{Prop}
\label{PropDavid}
Let $g\in\Sch_\omega(\R^d)$. Then the multiplication operator
\beqs
\label{52}
M_g:\ \calO_{M,\omega}(\R^d)&&\longmapsto\Sch_\omega(\R^d)\\
\nonumber
F&&\longmapsto g\cdot F
\eeqs
is continuous.
\end{Prop}

\begin{proof}
 By \cite[Theorem~5.4 (i)]{DN}, the space $\calO_{M,\omega}(\R^d)$ is ultrabornological, so the result follows from De Wilde closed graph theorem.
\end{proof}

We are now ready to prove stability of the Gabor $\omega$-wave front set by nonstationary perturbations of Gabor frames.

We recall, from \cite{BJO-Gabor}:
\begin{Def}
\label{defWF'}
Let $u\in \Sch'_\omega(\R^d)$ and $\varphi\in\Sch_\omega(\R^d)\setminus\{0\}$.
We say that $z_0\in\R^{2d}\setminus\{0\}$ is not in the $\omega$-wave front set 
$\WF'_\omega(u)$ of $u$ if there exists an open conic set 
$\Gamma\subseteq\R^{2d}\setminus\{0\}$ containing $z_0$ such that
\beqs
\label{13}
\sup_{z\in\Gamma}e^{\lambda\omega(z)}|V_\varphi u(z)|<+\infty,\qquad\forall\lambda>0.
\eeqs
\end{Def}
The $\omega$-wave front set is independent of the choice of the window function $\varphi$ by
\cite[Prop.~3.2]{BJO-Gabor}.
Let us now prove that it coincides with $\WF^{\{g_{m,n}\}}_\omega(u)$:
\begin{Th}
\label{th2}
Let $g\in\Sch_\omega(\R^d)\setminus\{0\}$ with
$\supp g\subseteq[A_1,B_1]\times\cdots\times[A_d,B_d]$. 
Set
 $g_n(t)=g(t-\alpha n)$ and $g_{m,n}(t)=e^{i\beta_n m\cdot t}g(t-\alpha n)$ for
$\alpha,\beta_n>0$ satisfying \eqref{7}, \eqref{10} and \eqref{11}.

Then, for $u\in\Sch'_\omega(\R^d)$,
\beqsn
\WF'_\omega(u)=\WF^{\{g_{m,n}\}}_\omega(u).
\eeqsn
\end{Th}

\begin{proof}
Let us first remark that
\beqsn
\langle u,g_{m,n}\rangle=V_gu(\alpha n,\beta_n m).
\eeqsn
Then clearly $\WF^{\{g_{m,n}\}}_\omega(u)\subseteq \WF'_\omega(u)$.

We have to prove the opposite inequality.
Let $0\neq z_0\notin\WF^{\{g_{m,n}\}}_\omega(u)$ and prove that $z_0\notin \WF'_\omega(u)$.
From Definition~\ref{y-sigmaWF} there exists an open conic set 
$\Gamma\subseteq\R^{2d}\setminus\{0\}$ containing $z_0$ such that 
\beqs
\label{16}
\sup_{\afrac{m,n\in\Z^d}{(\alpha n,\beta_n m)\in\Gamma}}
e^{\lambda\omega(\alpha n,\beta_n m)}|\langle u,g_{m,n}\rangle|<+\infty,
\qquad\forall\lambda>0.
\eeqs

From Proposition~\ref{lemma2}
\beqsn
\langle u,\psi\rangle=\sum_{m,n\in\Z^d}\langle u,g_{m,n}\rangle\langle\tilde{g}_{m,n},\psi\rangle
=\sum_{m,n\in\Z^d}V_gu(\alpha n,\beta_n m)\langle\tilde{g}_{m,n},\psi\rangle,
\qquad\forall\psi\in\Sch_\omega(\R^d),
\eeqsn
where $\tilde{g}_{m,n}$ is the canonical dual window of $g_{m,n}$ defined in \eqref{14}.

The idea is now to argue similarly as in \cite[Thm.~3.17]{BJO-Gabor}.
We denote
\beqsn
&&u_1:=\sum_{\afrac{m,n\in\Z^d}{(\alpha n,\beta_n m)\in\Gamma}}
V_gu(\alpha n,\beta_n m)\tilde{g}_{m,n}\\
&&u_2:=\sum_{\afrac{m,n\in\Z^d}{(\alpha n,\beta_n m)\notin\Gamma}}
V_gu(\alpha n,\beta_n m)\tilde{g}_{m,n}.
\eeqsn
Clearly $V_gu=V_gu_1+V_gu_2$ and we shall thus estimate $|V_gu_1|$ and $|V_gu_2|$
separately.
For every $\gamma,\delta\in\N_0^d$ and $\lambda,\mu>0$:
\beqs
\nonumber
&&e^{-\lambda\varphi^*\left(\frac{|\gamma|}{\lambda}\right)}
e^{-\mu\varphi^*\left(\frac{|\delta|}{\mu}\right)}|t^\delta\partial^\gamma u_1(t)|\\
\nonumber
\leq&&
\sum_{\afrac{m,n\in\Z^d}{\sigma=(\alpha n,\beta_n m)\in\Gamma}}
|V_gu(\sigma)|\cdot|t^\delta\partial^\gamma\tilde{g}_{m,n}(t)|
e^{-\lambda\varphi^*\left(\frac{|\gamma|}{\lambda}\right)}
e^{-\mu\varphi^*\left(\frac{|\delta|}{\mu}\right)}\\
\label{31}
\leq&&
\sum_{\afrac{m,n\in\Z^d}{\sigma=(\alpha n,\beta_n m)\in\Gamma}}
|V_gu(\sigma)|\cdot\left|t^\delta\partial^\gamma\left(\frac{g(t-\alpha n)}{G(t)}e^{i\beta_n m\cdot t}\right)\right|
e^{-\lambda\varphi^*\left(\frac{|\gamma|}{\lambda}\right)}
e^{-\mu\varphi^*\left(\frac{|\delta|}{\mu}\right)}
\eeqs
where $G(t)$ is defined in \eqref{30} and $1/G(t)\in\calO_{M,\omega}(\R^d)$ by
Proposition~\ref{lemma3}, i.e. \eqref{18} is satisfied.

We now define $\langle y\rangle:=\sqrt{1+|y|^2}$ for every $y\in\R^d$; we then have, from \eqref{31}:
\beqs
\nonumber
&&e^{-\lambda\varphi^*\left(\frac{|\gamma|}{\lambda}\right)}
e^{-\mu\varphi^*\left(\frac{|\delta|}{\mu}\right)}|t^\delta\partial^\gamma u_1(t)|\\
\nonumber
\leq&&
\sum_{\afrac{m,n\in\Z^d}{\sigma=(\alpha n,\beta_n m)\in\Gamma}}
|V_gu(\sigma)|\cdot|t|^{|\delta|}\sum_{\iota\leq\gamma}\binom{\gamma}{\iota}
\langle\beta_n m\rangle^{|\gamma-\iota|}
\left|\partial^\iota\frac{g(t-\alpha n)}{G(t)}\right|
e^{-\lambda\varphi^*\left(\frac{|\gamma|}{\lambda}\right)}
e^{-\mu\varphi^*\left(\frac{|\delta|}{\mu}\right)}\\
\nonumber
\leq&&\sum_{\afrac{m,n\in\Z^d}{\sigma=(\alpha n,\beta_n m)\in\Gamma}}
|V_gu(\sigma)|\sum_{\iota\leq\gamma}\binom{\gamma}{\iota}
\sum_{\kappa\leq\iota}\binom{\iota}{\kappa}2^{-|\gamma|}|t|^{|\delta|}
e^{-\mu\varphi^*\left(\frac{|\delta|}{\mu}\right)}
\langle\beta_n m\rangle^{|\gamma-\iota|}\\
\label{32}
&&\cdot
\left|\partial^\kappa\frac{1}{G(t)}\right|\cdot
\left|\partial^{\iota-\kappa}g(t-\alpha n)\right| 2^{|\gamma|}
e^{-\lambda\varphi^*\left(\frac{|\gamma|}{\lambda}\right)}.
\eeqs

We now recall (see, for instance, \cite[(2.12),(2.1)]{BJO-Gabor}):
\beqs
\label{33}
&&\forall\mu>0\ \exists C_\mu>0:\quad|t|^{|\delta|}
e^{-\mu\varphi^*\left(\frac{|\delta|}{\mu}\right)}
\leq C_\mu e^{\mu\omega(t)}\\
\label{34}
&&\forall\lambda>0\quad 2^{|\gamma|}
e^{-\lambda\varphi^*\left(\frac{|\gamma|}{\lambda}\right)}
\leq e^{-3\lambda\varphi^*\left(\frac{|\gamma|}{3\lambda}\right)}.
\eeqs
Moreover, from \eqref{18} we have that for all $\lambda'>0$ there exist $C_{\lambda'},\mu_{\lambda'}>0$ such that
\beqs
\label{35}
\left|D^\kappa\frac{1}{G(t)}\right|\leq C_{\lambda'}
e^{\lambda'\varphi^*\left(\frac{|\kappa|}{\lambda'}\right)}
e^{\mu_{\lambda'}\omega(t)},
\eeqs
and then, being $g\in\Sch_\omega(\R^d)$, from \cite[Thm.~2.4]{BJO-realPW} and
the subadditivity of $\omega$ we have that for all $\lambda',\mu>0$ there exists $\mu_{\lambda'}>0$ as in \eqref{35} and  $C_{\lambda',\mu}>0$ such that
\beqs
\nonumber
|\partial^{\iota-\kappa}g(t-\alpha n)|e^{(\mu+\mu_{\lambda'})\omega(t)}
\nonumber
\leq&&
|\partial^{\iota-\kappa}g(t-\alpha n)|e^{(\mu+\mu_{\lambda'})\omega(t-\alpha n)}
e^{(\mu+\mu_{\lambda'})\omega(\alpha n)}\\
\label{36}
\leq&&C_{\lambda',\mu}e^{\lambda'\varphi^*\left(\frac{|\iota-\kappa|}{\lambda'}\right)}
e^{(\mu+\mu_{\lambda'})\omega(\alpha n)}.
\eeqs
Substituting \eqref{33}-\eqref{36} into \eqref{32} we get:
\beqs
\nonumber
&&e^{-\lambda\varphi^*\left(\frac{|\gamma|}{\lambda}\right)}
e^{-\mu\varphi^*\left(\frac{|\delta|}{\mu}\right)}|t^\delta\partial^\gamma u_1(t)|\\
\nonumber
\leq&&\sum_{\afrac{m,n\in\Z^d}{\sigma=(\alpha n,\beta_n m)\in\Gamma}}
|V_gu(\sigma)|\sum_{\iota\leq\gamma}\binom{\gamma}{\iota}
\sum_{\kappa\leq\iota}\binom{\iota}{\kappa}2^{-|\gamma|}C_\mu
\langle\beta_n m\rangle^{|\gamma-\iota|}\\
\nonumber
&&\cdot C_{\lambda'}e^{\lambda'\varphi^*\left(\frac{|\kappa|}{\lambda'}\right)}C_{\lambda',\mu}
e^{\lambda'\varphi^*\left(\frac{|\iota-\kappa|}{\lambda'}\right)}
e^{(\mu+\mu_{\lambda'})\omega(\alpha n)}
e^{-3\lambda\varphi^*\left(\frac{|\gamma|}{3\lambda}\right)}\\
\nonumber
\leq&&C'_{\lambda',\mu}
\sum_{\afrac{m,n\in\Z^d}{\sigma=(\alpha n,\beta_n m)\in\Gamma}}
|V_gu(\sigma)|e^{(\mu+\mu_{\lambda'})\omega(\alpha n)}
\sum_{\iota\leq\gamma}\binom{\gamma}{\iota}
\sum_{\kappa\leq\iota}\binom{\iota}{\kappa}2^{-|\gamma|}2^{-|\iota|}\\
\label{37}
&&\cdot 2^{|\iota|}
e^{\lambda'\varphi^*\left(\frac{|\iota|}{\lambda'}\right)}
e^{-3\lambda\varphi^*\left(\frac{|\gamma|}{3\lambda}\right)}
\langle\beta_n m\rangle^{|\gamma-\iota|},
\eeqs
for some $C'_{\lambda',\mu}>0$.
From \cite[Lemma~A.1 (iv)]{BJO-realPW}
\beqsn
2^{|\iota|}
e^{\lambda'\varphi^*\left(\frac{|\iota|}{\lambda'}\right)}
\leq e^{\frac{\lambda'}{3}\varphi^*\left(\frac{|\iota|}{\lambda'/3}\right)},
\eeqsn
which can be substituted in \eqref{37} to get, for some $C_{\mu,\lambda'}>0$:
\beqs
\nonumber
&&e^{-\lambda\varphi^*\left(\frac{|\gamma|}{\lambda}\right)}
e^{-\mu\varphi^*\left(\frac{|\delta|}{\mu}\right)}|t^\delta\partial^\gamma u_1(t)|\\
\leq&&C_{\mu,\lambda'}
\sum_{\afrac{m,n\in\Z^d}{\sigma=(\alpha n,\beta_n m)\in\Gamma}}
|V_gu(\sigma)|e^{(\mu+\mu_{\lambda'})\omega(\alpha n)}
\sum_{\iota\leq\gamma}\binom{\gamma}{\iota}
\sum_{\kappa\leq\iota}\binom{\iota}{\kappa}2^{-|\gamma|}2^{-|\iota|}\\
\label{40}
&&\cdot e^{\frac{\lambda'}{3}\varphi^*\left(\frac{|\iota|}{\lambda'/3}\right)}
e^{-3\lambda\varphi^*\left(\frac{|\gamma|}{3\lambda}\right)}
\langle\beta_n m\rangle^{|\gamma-\iota|}.
\eeqs

For $\lambda'=18\lambda$ so that $\lambda'/3=6\lambda$, from 
\cite[Lemma~A.1 (ix)]{BJO-realPW} we have that
\beqsn
e^{\frac{\lambda'}{3}\varphi^*\left(\frac{|\iota|}{\lambda'/3}\right)}
e^{-3\lambda\varphi^*\left(\frac{|\gamma|}{3\lambda}\right)}
\leq e^{-6\lambda\varphi^*\left(\frac{|\gamma-\iota|}{6\lambda}\right)}
\eeqsn
which gives, together with $\langle\beta_n m\rangle^{|\gamma-\iota|}$, by \eqref{33}:
\beqsn
e^{\frac{\lambda'}{3}\varphi^*\left(\frac{|\iota|}{\lambda'/3}\right)}
e^{-3\lambda\varphi^*\left(\frac{|\gamma|}{3\lambda}\right)}
\langle\beta_n m\rangle^{|\gamma-\iota|}\leq C_{6\lambda} e^{6\lambda\omega(\langle\beta_n m\rangle)}.
\eeqsn
Substituting in \eqref{40}, and using \eqref{16} for $\langle u,g_{m,n}\rangle=V_gu(\alpha n,\beta_n m)$, we have that, for some constants depending on $\mu$ and $\lambda$
(with $\mu'_\lambda:=\mu_{\lambda'}=\mu_{18\lambda}$) and for $M>0$ big enough,
\beqsn
&&e^{-\lambda\varphi^*\left(\frac{|\gamma|}{\lambda}\right)}
e^{-\mu\varphi^*\left(\frac{|\delta|}{\mu}\right)}|t^\delta\partial^\gamma u_1(t)|\\
\leq&&\tilde{C}_{\mu,\lambda}
\sum_{\afrac{m,n\in\Z^d}{\sigma=(\alpha n,\beta_n m)\in\Gamma}}
|V_gu(\sigma)|e^{(\mu+\mu'_\lambda)\omega(\alpha n)}
\sum_{\iota\leq\gamma}\binom{\gamma}{\iota}
\sum_{\kappa\leq\iota}\binom{\iota}{\kappa}2^{-|\gamma|}2^{-|\iota|}
e^{6\lambda\omega(\langle\beta_n m\rangle)}\\
\leq&&\tilde{C}_{\mu,\lambda}
\sum_{\afrac{m,n\in\Z^d}{\sigma=(\alpha n,\beta_n m)\in\Gamma}}
|V_gu(\sigma)|e^{(\mu+\mu'_\lambda+6\lambda+M)\omega(\sigma)}
e^{-M\omega(\sigma)}\\
\leq&&\tilde{C}'_{\mu,\lambda}\sum_{m,n\in\Z^d}e^{-M\omega(\alpha n,\beta_n m)}
\leq \tilde{C}''_{\mu,\lambda}\,,
\eeqsn
since assumpiton $(\gamma)$ on $\omega$ implies that the series
\beqsn
\sum_{m,n\in\Z^d}e^{-M\omega(\alpha n,\beta_n m)}
\leq&&\sum_{m,n\in\Z^d}e^{-aM}e^{-bM\log(1+|(\alpha n,\beta_n m)|)}\\
\leq&&e^{-aM}\sum_{m,n\in\Z^d}\frac{1}{(1+|\alpha n|^2+
|\beta_n m|^2)^{\frac{bM}{2}}}
\eeqsn 
converges for large $M$, considering that $\{\beta_n\}_{n\in\N}$ is bounded by
\eqref{7} and \eqref{12}.

We have thus proved that $u_1\in\Sch_\omega(\R^d)$. Then, by
\cite[Thm.~2.7]{GZ}, also $V_g u_1\in\Sch_\omega(\R^d)$ and
\beqs
\label{324Gabor}
\forall\lambda>0\ \exists C_\lambda>0:\quad
e^{\lambda\omega(z)}|V_g u_1(z)|\leq C_\lambda\qquad\forall z\in\R^{2d}.
\eeqs

In order to estimate $|V_g u_2(z)|$, let us now fix an open conic set $\Gamma'\subseteq\R^{2d}\setminus\{0\}$
containing $z_0$ such that $\overline{\Gamma'\cap S_{2d-1}}\subseteq\Gamma$,
where $S_{2d-1}$ is the unit sphere in $\R^{2d}$.
Then
\beqsn
\inf_{\substack{m,n\in\Z^d\\0\neq\sigma=(\alpha n,\beta_n m)\in\R^{2d}\setminus\Gamma\\
z\in\Gamma'}}\left|\frac{\sigma}{|\sigma|}-z\right|
=\varepsilon>0
\eeqsn
and $|\sigma-z|\geq\varepsilon|\sigma|$ for $0\neq\sigma=(\alpha n,\beta_n m)\in\R^{2d}\setminus\Gamma$, $m,n\in\Z^d$, $z\in\Gamma'$.

Since $u\in\Sch'_\omega(\R^d)$ and $g\in\Sch_\omega(\R^d)$, from \cite[Thm.~2.4]{GZ} there
exist constants $\bar c,\bar\lambda>0$ such that
\beqsn
|V_g u(\zeta)|\leq\bar ce^{\bar\lambda\omega(\zeta)}\qquad\forall\zeta\in\R^{2d}.
\eeqsn
Then, for all $\lambda>0$, by the subadditivity of $\omega$ we have, for $z=(x,\xi)$:
\beqs
\nonumber
e^{\lambda\omega(z)}|V_gu_2(z)|\leq&&
\sum_{\afrac{m,n\in\Z^d}{\sigma=(\alpha n,\beta_n m)\notin\Gamma}}
e^{\lambda\omega(\sigma)}e^{\lambda\omega(z-\sigma)}
|V_gu(\alpha n,\beta_n m)|\cdot|\langle\tilde{g}_{m,n},\Pi(z)g\rangle|\\
\label{41}
\leq&&\bar c
\sum_{\afrac{m,n\in\Z^d}{\sigma=(\alpha n,\beta_n m)\notin\Gamma}}
e^{(\lambda+\bar\lambda)\omega(\sigma)}e^{\lambda\omega(z-\sigma)}
\Big|\Big\langle\frac{g(t-\alpha n)}{\sum_{\ell\in\Z^d}\frac{1}{\beta_\ell^d}|g_\ell(t)|^2}
e^{i\beta_n m\cdot t},e^{it\cdot\xi}g(t-x)\Big\rangle\Big|.
\eeqs

By the change of variables $\tau=t-\alpha n$ we can write
\beqs
\nonumber
&&\Big|\Big\langle\frac{g(t-\alpha n)}{\sum_{\ell\in\Z^d}\frac{1}{\beta_\ell^d}|g_\ell(t)|^2}
e^{i\beta_n m\cdot t},e^{it\cdot\xi}g(t-x)\Big\rangle\Big|\\
\nonumber
=&&\Big|\Big\langle\frac{g(\tau)}{\sum_{\ell\in\Z^d}\frac{1}{\beta_\ell^d}|g(\tau-\alpha(\ell-n))|^2}
e^{i\beta_n m\cdot (\tau+\alpha n)},e^{i(\tau+\alpha n)\cdot\xi}g(\tau+\alpha n-x)\Big\rangle\Big|\\
\nonumber
=&&\Big|\Big\langle\frac{g(\tau)}{\sum_{\ell\in\Z^d}\frac{1}{\beta_{\ell+n}^d}|g(\tau-\alpha\ell)|^2},
e^{i\tau\cdot(\xi-\beta_n m)}g(\tau-(x-\alpha n)\Big\rangle\Big|\\
\nonumber
=&&\Big|\Big\langle\frac{g(\tau)}{\sum_{\ell\in\Z^d}\frac{1}{\beta_{\ell+n}^d}|g_\ell(\tau)|^2},
\Pi(x-\alpha n,\xi-\beta_n m)g(\tau)\Big\rangle\Big|\\
\label{42}
=&&\Big|V_g\Big(\frac{g(\tau)}{\sum_{\ell\in\Z^d}\frac{1}{\beta_{\ell+n}^d}|g_\ell(\tau)|^2}\Big)
(x-\alpha n,\xi-\beta_n m)\Big|\\
\nonumber
=&&\Big|V_g\Big(\frac{g(\tau)}{\sum_{\ell\in\Z^d}\frac{1}{\beta_{\ell+n}^d}|g_\ell(\tau)|^2}\Big)
(z-\sigma)\Big|.
\eeqs

Let us now set
\beqsn
G_n(\tau)=\frac{1}{\sum_{\ell\in\Z^d}\frac{1}{\beta_{\ell+n}^d}|g_\ell(\tau)|^2}.
\eeqsn
From \eqref{7} and \eqref{12} the sequence $\{\beta_{\ell+n}\}_{\ell\in\Z^d}$ is uniformly bounded and, as in \eqref{23}, $\{G_n\}_{n\in\Z^d}$ is bounded in $\calO_{M,\omega}(\R^d)$.
From Proposition~\ref{PropDavid} the multiplication operator $M_g$ is continuous, hence 
sends bounded sets into bounded sets. It follows that $\{g/G_n\}_{n\in\Z^d}$ is bounded in
$\Sch_\omega(\R^d)$ and therefore, from \cite[Thm.~2.7]{GZ},
\beqs
\label{43}
\forall\mu>0\ \exists C_\mu>0:\quad
e^{\mu\omega(z-\sigma)}\left|V_g\left(\frac{g}{G_n}\right)(z-\sigma)\right|\leq C_\mu,\qquad\forall n\in\Z^d.
\eeqs

Applying the estimate \eqref{43} in \eqref{42} and then in \eqref{41}, we have that for all $\mu>0$ there exists $C_\mu>0$ such that
\beqs
\label{328Gabor}
e^{\lambda\omega(z)}|V_g u_2(z)|
\leq C_\mu
\sum_{\afrac{m,n\in\Z^d}{\sigma=(\alpha n,\beta_n m)\notin\Gamma}}
e^{(\lambda+\bar\lambda)\omega(\sigma)}e^{(\lambda-\mu)\omega(z-\sigma)}.
\eeqs

Taking into account that for $z\in\Gamma'$ and $\sigma=(\alpha n,\beta_n m)\in\R^{2d}\setminus\Gamma$ we have $|\sigma-z|\geq\varepsilon|\sigma|$, by the subadditivity of $\omega$ we have that
\beqsn
\omega(\sigma)\leq\omega\left(\frac{|\sigma-z|}{\varepsilon}\right)\leq
\left(\left[\frac1\varepsilon\right]+1\right)\omega(\sigma-z)
\eeqsn 
and hence choosing $\mu>\lambda$ sufficiently large we get from \eqref{328Gabor} that
\beqs
\label{329Gabor}
e^{\lambda\omega(z)}|V_g u_2(z)|
\leq C_\lambda
\sum_{\afrac{m,n\in\Z^d}{\sigma=(\alpha n,\beta_n m)\notin\Gamma}}
e^{(\lambda+\bar\lambda+([1/\varepsilon]+1)^{-1}(\lambda-\mu))\omega(\sigma)}
\leq C'_\lambda\qquad\forall z\in\Gamma'
\eeqs
for some $C_\lambda,C'_\lambda>0$.

From \eqref{324Gabor} and \eqref{329Gabor} we finally get that
\beqsn
\sup_{z\in\Gamma'}e^{\lambda\omega(z)}|V_gu(z)|<+\infty,\qquad\forall\lambda>0,
\eeqsn
and hence $z_0\notin\WF'_\omega(u)$.
\end{proof}

Being $\WF'_\omega(u)=\WF^G_\omega(u)$ by \cite[Thm.~3.17]{BJO-Gabor}, we finally have that the Gabor $\omega$-wave front set is stable by perturbing the Gabor frame 
$\{\Pi(\alpha n,\beta m)g\}_{m,n\in\Z^d}$ with $\{g_{m,n}\}_{m,n\in\Z^d}=\{\Pi(\alpha n,\beta_n m)g\}_{m,n\in\Z^d}$:

\begin{Cor}
\label{cor1}
For $g_{m,n}$ as in Theorem~\ref{th2}, if $u\in\Sch'_\omega(\R^d)$, then
\beqsn
\WF^G_\omega(u)=\WF^{\{g_{m,n}\}}_\omega(u).
\eeqsn
\end{Cor}

\begin{Rem}
\begin{em}
Since the Gabor $\omega$-wave front set is independent of the choice of the window function 
$g\in\Sch_\omega(\R^d)$, so is $\WF^{\{g_{m,n}\}}_\omega(u)$, which thus is independent also of the perturbation $\{\beta_n\}_{n\in\Z^d}$ of $\beta>0$, provided that $g$ satisfies the assumptions of Theorem~\ref{th2}.
\end{em}
\end{Rem}

Similar results may be obtained by adaptation over frequency, taking inspiration from \cite{Balazs}.
Let $h\in\Sch_\omega(\R^d)\setminus\{0\}$ with $\supp\hat h\subseteq[A_1,B_1]\times\dots[A_d,B_d]$. Set
$h_m(t)=e^{i\beta m\cdot t}h(t)$ and $h_{m,n}(t)=h_m(t-\alpha_m n)=e^{i\beta m\cdot(t-\alpha_m n)}
h(t-\alpha_m n)$ for $\alpha_m,\beta>0$ with
\beqs
\label{alpham}
\frac{1}{\alpha_m}\geq\max_{1\leq j\leq d}(B_j-A_j).
\eeqs
It will be useful to write
\beqs
\label{53}
h_{m,n}(t)=e^{-i\beta \alpha_m m\cdot n}\Pi(\alpha_m n,\beta m)h(t)
\eeqs
and recall, by the properties of the Fourier transform, that
\beqsn
&&\hat{h}_m(\xi)=\hat{h}(\xi-\beta m)\\
&&\hat{h}_{m,n}(\xi)=e^{-i\alpha_m n\cdot\xi}\hat{h}_m(\xi)=
e^{-i\alpha_m n\cdot\xi}\hat{h}(\xi-\beta m)=\Pi(\beta m,-\alpha_m n)\hat{h}(\xi).
\eeqsn

By Plancharel's Theorem, 
$\{h_{m,n}(t)\}_{m,n\in\Z^d}=\{e^{-i\beta\alpha_m m\cdot n}\Pi(\alpha_m n,\beta m)h(t)\}_{m,n\in\Z^d}$ is a frame in 
$L^2(\R^d)$ if and only if
$\{\hat{h}_{m,n}(\xi)\}_{m,n\in\Z^d}=\{\Pi(\beta m,-\alpha_m n)\hat{h}(\xi)\}_{m,n\in\Z^d}$ is a frame in $L^2(\R^d)$.
It follows, arguing similarly as for $\{g_{m,n}\}_{m,n\in\Z^d}$, that 
$\{h_{m,n}\}_{m,n\in\Z^d}$ is a frame if and only if
\beqs
\label{infalpha}
&&\inf_{t\in\R^d}\sum_{m\in\Z^d}\frac{1}{\alpha_m^d}|\hat{h}_m(\xi)|^2=A>0\\
\label{supalpha}
&&\sup_{t\in\R^d}\sum_{m\in\Z^d}\frac{1}{\alpha_m^d}|\hat{h}_m(\xi)|^2=B<+\infty.
\eeqs
Moreover, the canonical dual frame of $\{h_{m,n}\}_{m,n\in\Z^d}$ is given by (cf.
\cite[Cor.~2]{Balazs})
\beqsn
\tilde{h}_{m,n}(t)=T_{\alpha_m n}\F^{-1}\left(\frac{\hat h_m}{\sum_{\ell\in\Z^d}\frac{1}{\alpha_\ell^d}|\hat h_\ell|^2}\right)(t)
\eeqsn
and
\beqs
\label{45}
\hat{\tilde h}_{m,n}(\xi)=e^{-i\alpha_m n\cdot\xi}
\frac{\hat h_m(\xi)}{\sum_{\ell\in\Z^d}\frac{1}{\alpha_\ell^d}|\hat h_\ell(\xi)|^2}.
\eeqs

We prove now the following stability result on the Gabor $\omega$-wave front set, similarly as in Theorem~\ref{th2} and Corollary~\ref{cor1}:

\begin{Th}
\label{th3}
Let $h\in\Sch_\omega(\R^d)\setminus\{0\}$ with $\supp\hat h\subseteq[A_1,B_1]\times\dots\times [A_d,B_d]$. Set
$h_m(t)=e^{i\beta m\cdot t}h(t)$ and $h_{m,n}(t)=h_m(t-\alpha_m n)$ for $\alpha_m,\beta>0$ satisfying \eqref{alpham}, \eqref{infalpha} and \eqref{supalpha}.

Then, for $u\in\Sch'_\omega(\R^d)$,
\beqsn
\WF'_\omega(u)=\WF^G_\omega(u)=\WF^{\{h_{m,n}\}}_\omega(u).
\eeqsn
\end{Th}

\begin{proof}
We just have to prove the inclusion
\beqsn
\WF'_\omega(u)\subseteq\WF^{\{h_{m,n}\}}_\omega(u).
\eeqsn
Let $0\neq z_0\notin\WF^{\{h_{m,n}\}}_\omega(u)$. There exists then an open conic set $\Gamma\subset\R^{2d}\setminus\{0\}$ containing $z_0$ such that
\beqsn
\sup_{(\alpha_m n,\beta m)\in\Gamma}e^{\lambda\omega(\alpha_m n,\beta m)}
|\langle u,h_{m,n}\rangle|<+\infty,\qquad
\forall\lambda>0.
\eeqsn
For all $\psi\in\Sch_\omega(\R^d)$, from \eqref{53} we have:
\beqsn
\langle u,\psi\rangle=&&\sum_{m,n\in\Z^d}\langle u,h_{m,n}\rangle
\langle\tilde{h}_{m,n},\psi\rangle\\
=&&\sum_{m,n\in\Z^d}e^{i\beta\alpha_m m\cdot n}\langle u,\Pi(\alpha_m n,\beta m)h\rangle
\langle\tilde{h}_{m,n},\psi\rangle\\
=&&\sum_{m,n\in\Z^d}e^{i\beta\alpha_m m\cdot n}V_hu(\alpha_m n,\beta m)
\langle\tilde{h}_{m,n},\psi\rangle.
\eeqsn

Let us set
\beqs
\label{47}
&&u_1:=\sum_{\afrac{m,n\in\Z^d}{\sigma=(\alpha_m n,\beta m)\in\Gamma}}
e^{i\beta\alpha_m m\cdot n}V_hu(\sigma)
\tilde{h}_{m,n}\\
\label{48}
&&u_2:=\sum_{\afrac{m,n\in\Z^d}{\sigma=(\alpha_m n,\beta m)\notin\Gamma}}
e^{i\beta\alpha_m m\cdot n}V_hu(\sigma)
\tilde{h}_{m,n}.
\eeqs
Clearly $V_hu=V_hu_1+V_hu_2$. Let us first prove that $V_hu_1\in\Sch_\omega(\R^d)$.
To this aim we prove that $V_{\hat h}\hat{u}_1\in\Sch_\omega(\R^d)$, recalling that
(see \cite[(3.10)]{G})
\beqsn
V_hu_1(x,\xi)=e^{-ix\cdot\xi} V_{\hat h}\hat{u}_1(\xi,-x).
\eeqsn

From \eqref{47} and \eqref{45}:
\beqsn
\hat{u}_1(\xi)=&&\sum_{\afrac{m,n\in\Z^d}{\sigma=(\alpha_m n,\beta m)\in\Gamma}}
e^{i\beta\alpha_m m\cdot n}V_hu(\sigma)
\hat{\tilde{h}}_{m,n}(\xi)\\
=&&\sum_{\afrac{m,n\in\Z^d}{\sigma=(\alpha_m n,\beta m)\in\Gamma}}
e^{i\beta\alpha_m m\cdot n}V_hu(\sigma)
e^{-i\alpha_m n\cdot\xi}
\frac{\hat h_m(\xi)}{\sum_{\ell\in\Z^d}\frac{1}{\alpha_\ell^d}|\hat h_\ell(\xi)|^2}.
\eeqsn

We can thus proceed similarly as for the estimate of 
$e^{-\lambda\varphi^*\left(\frac{|\gamma|}{\lambda}\right)}
e^{-\mu\varphi^*\left(\frac{|\delta|}{\mu}\right)}|t^\delta\partial^\gamma u_1(t)|$
in the proof of Theorem~\ref{th2}, to prove that $\hat{u}_1\in\Sch_\omega(\R^d)$.
Then $V_{\hat h}\hat{u}_1\in\Sch_\omega(\R^d)$ by \cite[Thm.~2.7]{GZ} and hence $V_h u_1\in\Sch_\omega(\R^d)$ and
\beqs
\label{50}
\forall\lambda>0\ \exists C_\lambda>0:\quad
e^{\lambda\omega(z)}|V_hu_1(z)|\leq C_\lambda\qquad\forall z\in\R^{2d}.
\eeqs

We now fix an open conic set $\Gamma'\subseteq\R^{2d}\setminus\{0\}$ as in the proof of Theorem~\ref{th2} and estimate $V_hu_2$ in $\Gamma'$.
From \eqref{48} and the subadditivity of $\omega$ we have
\beqs
\nonumber
e^{\lambda\omega(z)}|V_hu_2(z)|\leq && e^{\lambda\omega(z)}
\sum_{\afrac{m,n\in\Z^d}{\sigma=(\alpha_m n,\beta m)\notin\Gamma}}
|V_hu(\sigma)|
|\langle\tilde{h}_{m,n},\Pi(z)h\rangle|\\
\label{51}
\leq&&\sum_{\afrac{m,n\in\Z^d}{\sigma=(\alpha_m n,\beta m)\notin\Gamma}}
e^{\lambda\omega(\sigma)} |V_hu(\sigma)|
e^{\lambda\omega(z-\sigma)}|\langle\tilde{h}_{m,n},\Pi(z)h\rangle|.
\eeqs
For $z=(t,\xi)$:
\beqsn
|\langle\tilde{h}_{m,n},\Pi(z)h\rangle|=&&
\Big|\Big\langle T_{\alpha_m n}\F^{-1}\Big(\frac{\hat h_m}{\sum_{\ell\in\Z^d}\frac{1}{\alpha_\ell^d}|\hat{h}_\ell |^2}\Big)(t),e^{it\cdot\xi}h(t-x)\Big\rangle\Big|\\
=&&\Big|\Big\langle \F^{-1}\Big(\frac{\hat h_m}{\sum_{\ell\in\Z^d}\frac{1}{\alpha_\ell^d}|\hat{h}_\ell |^2}\Big)(t-\alpha_m n),e^{it\cdot\xi}h(t-x)\Big\rangle\Big|.
\eeqsn
By the change of variables $\tau=t-\alpha_m n$:
\beqsn
|\langle\tilde{h}_{m,n},\Pi(z)h\rangle|
=&&\Big|\Big\langle \F^{-1}\Big(\frac{\hat h_m}{\sum_{\ell\in\Z^d}\frac{1}{\alpha_\ell^d}|\hat{h}_\ell |^2}\Big)(\tau),e^{i(\tau+\alpha_m n)\cdot\xi}h(\tau+\alpha_m n-x)\Big\rangle\Big|\\
=&&\Big|\Big\langle \F^{-1}\Big(\frac{\hat h_m}{\sum_{\ell\in\Z^d}\frac{1}{\alpha_\ell^d}|\hat{h}_\ell |^2}\Big)(\tau),e^{i\tau\cdot\xi}h(\tau-(x-\alpha_m n))\Big\rangle\Big|\\
=&&\Big|\Big\langle \F^{-1}\Big(\frac{\hat h(\cdot-\beta m)}{\sum_{\ell\in\Z^d}\frac{1}{\alpha_\ell^d}|\hat{h}_\ell(\cdot) |^2}\Big)(\tau),e^{i\tau\cdot\xi}h(\tau-(x-\alpha_m n))\Big\rangle\Big|\\
=&&\Big|\Big\langle \F^{-1}\Big(\frac{\hat h(\cdot-\beta m)}{\sum_{\ell\in\Z^d}\frac{1}{\alpha_\ell^d}|\hat{h}_\ell(\cdot-\beta m+\beta m) |^2}\Big)(\tau),e^{i\tau\cdot\xi}h(\tau-(x-\alpha_m n))\Big\rangle\Big|\\
=&&\Big|\Big\langle \F^{-1}\Big(\frac{\hat h(\cdot-\beta m)}{\sum_{\ell\in\Z^d}\frac{1}{\alpha_\ell^d}|(T_{-\beta m}\hat{h}_\ell)(\cdot-\beta m) |^2}\Big)(\tau),e^{i\tau\cdot\xi}h(\tau-(x-\alpha_m n))\Big\rangle\Big|.
\eeqsn
Taking then into account that $\F^{-1}(f(\cdot-\beta m))(\tau)=e^{i\beta m\cdot\tau}\F^{-1}(f)(\tau)$
and $T_{-\beta m}\hat h_\ell=T_{-\beta m}\hat h(\cdot-\beta\ell)=\hat h(\cdot-\beta(\ell-m))$:

\beqsn
|\langle\tilde{h}_{m,n},\Pi(z)h\rangle|
=&&\Big|\Big\langle e^{i\beta m\cdot\tau}\F^{-1}\Big(\frac{\hat h}{\sum_{\ell\in\Z^d}\frac{1}{\alpha_\ell^d}|(T_{-\beta m}\hat{h}_\ell) |^2}\Big)(\tau),e^{i\tau\cdot\xi}h(\tau-(x-\alpha_m n))\Big\rangle\Big|\\
=&&\Big|\Big\langle e^{i\beta m\cdot\tau}\F^{-1}\Big(\frac{\hat h}{\sum_{\ell\in\Z^d}\frac{1}{\alpha_\ell^d}|\hat{h}(\cdot-\beta(\ell-m)) |^2}\Big)(\tau),e^{i\tau\cdot\xi}h(\tau-(x-\alpha_m n))\Big\rangle\Big|\\
=&&\Big|\Big\langle e^{i\beta m\cdot\tau}\F^{-1}\Big(\frac{\hat h}{\sum_{\ell\in\Z^d}\frac{1}{\alpha_{\ell+m}^d}|\hat{h}(\cdot-\beta\ell) |^2}\Big)(\tau),e^{i\tau\cdot\xi}h(\tau-(x-\alpha_m n))\Big\rangle\Big|\\
=&&\Big|\Big\langle \F^{-1}\Big(\frac{\hat h}{\sum_{\ell\in\Z^d}\frac{1}{\alpha_{\ell+m}^d}|
\hat{h}_\ell |^2}\Big)(\tau),e^{i\tau(\xi-\beta m)}h(\tau-(x-\alpha_m n))\Big\rangle\Big|\\
=&&\Big|V_h\Big(\F^{-1}\Big(\frac{\hat h}{\sum_{\ell\in\Z^d}\frac{1}{\alpha_{\ell+m}^d}|\hat{h}_\ell |^2}\Big)\Big)(x-\alpha_m n,\xi-\beta m)\Big|.
\eeqsn

We can thus proceed as in the estimate of \eqref{42}, by means of the continuity of the multiplication operator \eqref{52} and of the inverse Fourier transform $\F^{-1}$, to get from \eqref{51} that
\beqsn
\forall\lambda\ \exists C_\lambda>0:\quad
e^{\lambda\omega(z)}|V_hu_2(z)|\leq C_\lambda\quad\forall z\in\Gamma'.
\eeqsn

The above estimate together with \eqref{50} finally proves that $z_0\notin\WF'_\omega(u)$,
and the proof is complete.
\end{proof}

Note that, similarly as in the case of $\{g_{m,n}(t)\}_{m,n\in\Z^d}$, the frame 
$\{h_{m,n}(t)\}_{m,n\in\Z^d}$ is not an $\varepsilon$-perturbation of the Gabor frame
$\{\Pi(\alpha n,\beta m)h(t)\}_{m,n\in\Z^d}$ either.


\vspace{3mm}
{\bf Acknowledgments.}
Boiti was partially supported by the Projects FIRD 2022, FAR 2023, FIRD 2024 (University of Ferrara) and by the Italian Ministry of University and Research, under 
 PRIN 2022 (Scorrimento), project "Anomalies in partial differential equations and applications", code 2022HCLAZ8\_002, CUP J53C24002560006.
Boiti and Oliaro were partially supported  by
the Research Project GNAMPA-INdAM 2024 ``Analisi di Gabor ed
analisi microlocale: connessioni e applicazioni a equazioni a derivate parziali".
Jornet was partially supported by the Project PID2024-162128NB-I00 funded by MICIU /AEI /10.13039/501100011033 / FEDER, UE and by the Project CIAICO/2023/242 of the Conselleria de Innovaci\'on, Universidades, Ciencia y Sociedad Digital of Generalitat Valenciana.



\begin{thebibliography}{AA}

\bibitem{AJO1}
A.~A. Albanese, D.~Jornet, and A.~Oliaro, \emph{Quasianalytic wave front sets
	for solutions of linear partial differential operators}, Integral Equ.
Oper. Theory \textbf{66}, no.~2 (2010), 153--181.

\bibitem{AJO2}
A.~A. Albanese, D.~Jornet, and A.~Oliaro, 
\emph{Wave front sets for ultradistribution solutions of linear
	partial differential operators with coefficients in non-quasianalytic
	classes}, Math. Nachr. \textbf{285}, no.~4 (2012), 411--425.

\bibitem{AM-multipliers}
A.A.~Albanese, C.~Mele, {\em Multipliers in $\Sch(\R^N)$}, J. Pseudo-Differ. Oper. Appl. 
{\bf 12}, n.2 (2021), 35.

\bibitem{AM-spectra}
A.A.~Albanese, C.~Mele, {\em Spectra and ergodic properties of multiplication and convolution operators on the space $\Sch(\R)$}, Rev. Mat. Complut. {\bf 35} (2022), 739-762.

\bibitem{AFGLS}
W.~Alharbi, D.~Freeman, D.~Ghoreishi, C.~Lois, S.~Sebastian,
{\em Stable phase retrieval and perturbations of frames}, Proc. Amer. Math. Soc. Ser. B
{\bf 10} (2023), 353-368.

\bibitem{A-Q}
V.~Asensio, \emph{Quantizations and global hypoellipticity for
	pseudodifferential operators of infinite order in classes of
	ultradifferentiable functions}, Mediterr. J. Math. {\bf 19}, no.~3 (2022), Paper No. 135, 36 pp.

\bibitem{As}
V.~Asensio, \emph{Matrix-Wigner global wave front sets in ultradifferentiable classes}, J. Pseudo-Differ. Oper. Appl. {\bf 16} (2025), no.~1, Paper No. 11.

\bibitem{AJ}
V.~Asensio and D.~Jornet, \emph{Global pseudodifferential operators of infinite
	order in classes of ultradifferentiable functions}, Rev. R. Acad. Cienc.
Exactas F\'{\i}s. Nat. Ser. A Mat. RACSAM \textbf{113}, no.~4 (2019),
3477--3512.

\bibitem{ABJO-cptWeyl}
V.~Asensio, C.~Boiti, D.~Jornet, A.~Oliaro, {\em On the compactness of the Weyl  operator in $\mathcal S_\omega$},  
J. Math. Anal. Appl. {\bf  546}, n.1 (2025), 129214.

\bibitem{ABJO-global}
V.~Asensio, C.~Boiti, D.~Jornet, A.~Oliaro, {\em Global wave front sets in ultradifferentiable classes}, Results Math. {\bf 77}, no.~2 (2022), Paper No. 65.

\bibitem{A}
S.~Axler, {\em Linear Algebra Done Right}, $4^{th}$ Edition, Springer (2025).

\bibitem{Balazs}
P.~Balazs, M.~D\"orfler, F.~Jaillet, N.~Holighaus, G.~Velasco,
{\em Theory, implementation and applications of nonstationary Gabor frames}, 
J. Comput. Appl. Math. {\bf 236} (2011), 1481-1496.

\bibitem{Bjorck}
G.~Bj\"orck, {\em Linear partial differential operators and generalized distributions}, Ark. Mat. 
{\bf 6}, n.21 (1966), 351-407.

\bibitem{BJ-wfs}
C.~Boiti and D.~Jornet, \emph{A characterization of the wave front set defined by the iterates
	of an operator with constant coefficients}, Rev. R. Acad. Cienc. Exactas
F\'{\i}s. Nat. Ser. A Mat. RACSAM \textbf{111} (2017), no.~3, 891--919.

\bibitem{BJJ}
C.~Boiti, D.~Jornet, and J.~Juan-Huguet, \emph{Wave front sets with respect to
	the iterates of an operator with constant coefficients}, Abstr. Appl. Anal.
(2014), Art. ID 438716, 17.

\bibitem{BJO-Regularity}
C.~Boiti, D.~Jornet, and A.~Oliaro, \emph{Regularity of partial differential
	operators in ultradifferentiable spaces and {W}igner type transforms}, J.
Math. Anal. Appl. \textbf{446} (2017), no.~1, 920--944.

\bibitem{BJO-Gabor}
C.~Boiti, D.~Jornet, A.~Oliaro, {\em The Gabor wave front set in spaces of 
ultradifferentiable functions}, 
Monatsh. Math. {\bf 188}, n.2 (2019), 199-246.

\bibitem{BJO-realPW}
C.~Boiti, D.~Jornet, A.~Oliaro, {\em Real Paley-Wiener theorems in spaces of 
ultradifferentiable functions}, 
J. Funct. Anal. {\bf 278}, n.4 (2020), 108348.

\bibitem{boiti2020nuclearity}
C.~Boiti, D.~Jornet, A.~Oliaro, and G.~Schindl, \emph{Nuclearity of rapidly
	decreasing ultradifferentiable functions and time-frequency analysis},
Collect. Math. 72 (2021), no. 2, 423--442.

\bibitem{BM}
C.~Boiti, R.~Manfrin, {\em A direct and elementary derivation of the multivariate Fa\`a di Bruno formula}, to appear in Anal. Math.

\bibitem{BMT}
R.~W. Braun, R.~Meise, and B.~A. Taylor, \emph{Ultradifferentiable functions
	and {F}ourier analysis}, Results Math. \textbf{17} (1990), no.~3-4, 206--237.

\bibitem{CS}
M.~Cappiello and R.~Schulz, \emph{Microlocal analysis of quasianalytic
	{G}elfand-{S}hilov type ultradistributions}, Complex Var. Elliptic Equ.
\textbf{61} (2016), no.~4, 538--561.

\bibitem{C}
O.~Christensen, {\em Frame perturbations}, Proc. Amer. Math. Soc. {\bf 123}, n.4 (1995),
1217-1220.

\bibitem{CR1}
E.~Cordero and L.~Rodino, Time-Frequency Analysis of Operators, De Gruyter Studies in Mathematics 75, De Gruyter, Berlin, 2020.

\bibitem{CR2}
E.~Cordero and L.~Rodino, \emph{Wigner analysis of operators. Part I: Pseudodifferential operators and wave fronts}, Appl. Comput. Harmon. Anal. {\bf 58} (2022), 85--123.  

\bibitem{DN}  A.~Debrouwere, L.~Neyt, {\em Weighted (PLB)-spaces of ultradifferentiable functions and multiplier spaces}, Monatsh Math {\bf 198}  (2022), 31-60.


\bibitem{FG}
C.~Fern\'andez, A.~Galbis, {\em Superposition in Classes
of Ultradifferentiable Functions}, Publ. RIMS, Kyoto Univ.
{\bf 42} (2006), 399-419.

\bibitem{FGJ3}
C.~Fern\'{a}ndez, A.~Galbis, and D.~Jornet, \emph{Pseudodifferential operators of {B}eurling type and the wave
	front set}, J. Math. Anal. Appl. \textbf{340} (2008), no.~2, 1153--1170.

\bibitem{Fieker}
 C.~Fieker, {\em P-Konvexit\"at und $\omega$-Hypoelliptizit\"at f\"ur partielle Differentialoperatoren mit konstanten Koeffizienten}, Diplomarbeit, Mathematischen Institut der Heinrich-Heine-Universit\"at, D\"usseldorf, 1993.

\bibitem{G}
K.~Gr\"ochenig, {\em Foundations of Time-Frequency Analysis},
Birkh\"auser, Boston  (2001).

\bibitem{GZ}
K.~Gr\"ochenig, G.~Zimmermann, {\em Spaces of test functions via the STFT}, J. Funct. Spaces Appl. {\bf 2}, n.1 (2004), 25-53.


\bibitem{H5}
L.~H\"{o}rmander, \emph{Fourier integral operators. {I}}, Acta Math.
\textbf{127} (1971), no.~1-2, 79--183.

\bibitem{H3}
L.~H\"{o}rmander, \emph{Quadratic hyperbolic operators}, Microlocal analysis and
applications ({M}ontecatini {T}erme, 1989), Lecture Notes in Math., vol.
1495, Springer, Berlin, 1991, pp.~118--160.

\bibitem{MO}
C.~Mele and A.~Oliaro, \emph{Regularity of global solutions of partial differential equations in non isotropic ultradifferentiable spaces via time-frequency methods}, J. Differential Equations {\bf 286} (2021), 821--855.

\bibitem{N}
S.~Nakamura, \emph{Propagation of the homogeneous wave front set for
	{S}chr\"{o}dinger equations}, Duke Math. J. \textbf{126} (2005), no.~2,
349--367.

\bibitem{NR}
F.~Nicola and L.~Rodino, \emph{Global pseudo-differential calculus on
	{E}uclidean spaces}, Pseudo-Differential Operators. Theory and Applications,
vol.~4, Birkh\"{a}user Verlag, Basel, 2010.

\bibitem{PP}
S.~Pilipovi\'{c} and B.~Prangoski, \emph{Anti-{W}ick and {W}eyl quantization on
	ultradistribution spaces}, J. Math. Pures Appl. (9) \textbf{103} (2015),
no.~2, 472--503.

\bibitem{Pr}
B.~Prangoski, \emph{Pseudodifferential operators of infinite order in spaces of
	tempered ultradistributions}, J. Pseudo-Differ. Oper. Appl. \textbf{4}
(2013), no.~4, 495--549.

\bibitem{Ro}
L.~Rodino, \emph{Linear partial differential operators in {G}evrey spaces},
World Scientific Publishing Co., Inc., River Edge, NJ, 1993.

\bibitem{RW}
L.~Rodino and P.~Wahlberg, \emph{The {G}abor wave front set}, Monatsh. Math.
\textbf{173} (2014), no.~4, 625--655.

\bibitem{RW1}
L.~Rodino and P.~Wahlberg, \emph{Anisotropic global microlocal analysis for tempered distributions}, Monatsh. Math. {\bf 202} (2023), no.~2, 397--434.

\bibitem{SW2}
R.~Schulz and P.~Wahlberg, \emph{Equality of the homogeneous and the {G}abor
	wave front set}, Comm. Partial Differential Equations \textbf{42} (2017),
no.~5, 703--730.

\bibitem{W}
P.~Wahlberg, \emph{Propagation of anisotropic Gabor wave front sets}, Proc. Edinb. Math. Soc. (2) {\bf 67} (2024), no.~3, 674--698.

\end{thebibliography}
\end{document}